\documentclass[11pt, notitlepage]{article}
\usepackage[latin1]{inputenc}
\usepackage{amsfonts}
\usepackage{enumitem}
\usepackage{color}
\usepackage{textcomp}
\usepackage [english]{babel}
\usepackage{graphicx}
\usepackage{amssymb,amsmath,amsthm}
\usepackage[left=3cm,top=2.5cm,bottom=3cm,right=3cm]{geometry}
\usepackage[notcite,notref]{showkeys}
\newtheorem{define}{Definition}[section]
\newtheorem{pro}{Proposition}[section]
\newtheorem{Lem}{Lemma}[section]

\theoremstyle{plain} \newtheorem{thm}{Theorem}[section]
\newcommand{\eqdef}{\stackrel{\mathrm{def}}{=}}   

\catcode`\@=11
\makeatletter

\@addtoreset{equation}{section}

\begin{document}
\title{Adams-Moser-Trudinger inequality in the cartesian product of Sobolev spaces and its applications}
\author{~~R. Arora,~~J. Giacomoni\footnote{LMAP (UMR E2S-UPPA CNRS 5142) Bat. IPRA, Avenue de l'Universit\'e F-64013 Pau, France. email: rakesh.arora@univ-pau.fr, jacques.giacomoni@univ-pau.fr}, \ T. Mukherjee \footnote{Department of Mathematics, National Institute of Technology, Warangal, India. e-mail: tulimukh@gmail.com} and ~K. Sreenadh\footnote{Department of Mathematics, Indian Institute of Technology Delhi, Hauz Khaz, New Delhi-110016, India. e-mail: sreenadh@gmail.com} }
\maketitle
\begin{abstract}
The main aim of this article is to study non-singular version of Moser-Trudinger and Adams-Moser-Trudinger inequalities and the singular version of Moser-Trudinger equality in the Cartesian product of Sobolev spaces. As an application of these inequalities, we  study a system of Kirchhoff equations with exponential non-linearity of Choquard type. \vspace{.5cm} 
\vspace{.2cm}
\end{abstract}
\vspace{1cm}
\section{Introduction and main results}
Let $\Omega$ be a smooth bounded domain in $\mathbb{R}^n.$ Then the classical Sobolev space embedding says that
$$W_0^{1,p}(\Omega) \hookrightarrow L^{p^*}(\Omega) \ \text{if} \ n>p\;\ \text{where} \ p^*= \frac{np}{n-p} $$
or equivalently
$$ \sup_{\|u\|_{W^{1,p}_0(\Omega)} \leq 1} \int_\Omega |u|^r < \infty \ \text{for all} \ 1 \leq r \leq p^* \  \text{where} \ \|u\|^p_{W_0^{1,p}(\Omega)}= \int_{\Omega} |\nabla u|^p dx$$
 and in the limiting case $p=n$, $W_0^{1,n}(\Omega) \hookrightarrow L^{r}(\Omega)$ for all $1 \leq r < \infty$ but  not embedded in $L^{\infty}(\Omega).$ The maximal exponent $p^*$ is called as Sobolev  critical exponent. Hence, a natural question in connection with Orlicz space embeddings is to find a function $\phi: \mathbb{R} \to \mathbb{R}^+$ with maximal growth such that
$$\sup_{\|u\|_{W^{1,n}_0(\Omega)} \leq 1} \int_{\Omega} \phi(u)dx < \infty.$$
In this connection, in 1960's, Pohozaev \cite{Poho} and Trudinger \cite{trudinger} independently answered the question using the above  function with $\phi(t)= \exp(|t|^{\frac{n}{n-1}})-1.$ Later on, in \cite{Moser}, Moser improved the result by proving the following inequality which is popularly known as the  Moser-Trudinger inequality:
\begin{thm}\label{MT}
For $n\geq 2$, $\Omega \subset \mathbb R^n$ is a bounded domain and $u \in W^{1,n}_0(\Omega)$,
\[\sup_{\|u\|_{W_0^{1,n}(\Omega)}\leq 1}\int_\Omega \exp(\alpha|u|^{\frac{n}{n-1}})dx < \infty\]
if and only if $\alpha \leq \alpha_n$, where $\alpha_n = n\omega_{n-1}^{\frac{1}{n-1}}$ and $\omega_{n-1}=$ $(n-1)$- dimensional surface area of $\mathbb{S}^{n-1}$.
\end{thm}
\noindent Consequently, Adams \cite{Adams} extended the Moser's inequality to higher order Sobolev spaces by proving the following inequality which is known as Adams-Moser-Trudinger inequality:
\begin{thm}\label{TM-ineq}
Let $\Omega$ be a bounded domain in $\mathbb{R}^n$ and $n,m\in \mathbb N$ satisfying $m<n$. Then for all $0 \leq \zeta \leq \zeta_{n,m}$ and $u\in W_0^{m, \frac{n}{m}}(\Omega)$ we have
\[\sup_{\|\nabla^m u \|_{L^{\frac{n}{m}}(\Omega)} \leq 1}\int_\Omega \exp(\zeta |u|^{\frac{n}{n-m}})dx < \infty,\]
where $\zeta_{n,m}$ is sharp and given by
\begin{equation*}
\zeta_{n,m}=\left\{
\begin{split}
&\frac{n}{\omega_{n-1}}\left(\frac{\pi^{n/2}2^m \Gamma\left(\frac{m+1}{2}\right)}{\Gamma\left(\frac{n-m+1}{2}\right)}\right)^{\frac{n}{n-m}} \; \ \ \text{when}\ m \ \text{is odd},\\
&\frac{n}{\omega_{n-1}}\left(\frac{\pi^{n/2}2^m \Gamma\left(\frac{m}{2}\right)}{\Gamma\left(\frac{n-m}{2}\right)}\right)^{\frac{n}{n-m}} \; \ \ \ \ \  \text{when}\ m \ \text{is even.}
\end{split}
\right.
\end{equation*}
The symbol $\nabla^m u $ denotes the $m^{\text{th}}$-order gradient of $u$ and is defined as
\begin{equation*}
\nabla^m u=\left\{
\begin{split}
&\nabla \Delta^{(m-1)/2}u \; \  \text{if}\ m \ \text{is odd},  \\
&\Delta^{m/2} u \;\ \ \ \ \ \ \ \ \text{if}\ m \ \text{is even}
\end{split}
\right.
\end{equation*}
where $\Delta$ and $\nabla$ denotes the usual Laplacian and gradient operators respectively.
\end{thm}
\noindent Using the interpolation of Hardy inequality and Moser-Trudinger inequality, Adimurthi-Sandeep \cite{AdiSan} established the singular Moser-Trudinger inequality for  functions in $W_0^{1,n}(\Omega)$. This was consequently extended by Lam-Lu \cite{LamLu} for functions in $W_0^{m, \frac{n}{m}}(\Omega)$ while proving the following singular Adams-Moser-Trudinger inequality.
\begin{thm}\label{TM-ineq1}
Let $0\leq \alpha <n$, $\Omega$ be a bounded domain in $\mathbb{R}^n$ and $n,m \in \mathbb N$ satisfying $m<n$. Then for all $0 \leq \kappa \leq \kappa_{\alpha,n,m} = \left(1-\frac{\alpha}{n}\right)\zeta_{n,m}$ we have
\begin{equation}\label{AM}
\sup_{u \in W_0^{m, \frac{n}{m}}(\Omega),\ \|\nabla^m u\|_{L^{\frac{n}{m}}(\Omega)}\leq 1}\int_\Omega \frac{\exp(\kappa |u|^{\frac{n}{n-m}})}{|x|^\alpha}dx < \infty.
\end{equation}
If $\kappa > \kappa_{\alpha,n,m}$ then the above supremum is infinite ({\it i.e.} $\kappa_{\alpha,n,m}$ is sharp).
\end{thm}
In recent years, numerous generalizations, extensions and applications of the Moser-Trudinger and Adams-Trudinger-Moser inequalities have been widely explored and studied. A vast  amount of literature is available which are devoted to study these kinds of inequalities. We refer readers to  \cite{AdiSan, Adams, LamLu, Moser} for such topics and the survey article \cite{LamLu1} including the references within. In the field of geometric analysis curvature and partial differential equations where the nonlinear term behaves like $\exp\left( |t|^{\frac{n}{n-m}} \right)$ as $t \to \infty$, these inequalities play a vital role to carry out the analysis. \vspace{.2cm} \\
Motivated by the wide interest in the current literature, the aim of this paper is to answer the question of maximal growth function in Cartesian product of Sobolev spaces and establish both Moser-Trudinger and  Adams-Moser trudinger inequality alongwith the singular Moser-Trudinger inequality in the Cartesian product of Sobolev spaces for $n,m\in \mathbb N$ such that $n \geq 2m$. Let $$\mathcal{Y}:= W_0^{m,\frac{n}{m}}(\Omega) \times W_0^{m,\frac{n}{m}}(\Omega)$$
be the Banach space endowed with the norm $$\|(u,v)\|_{\mathcal{Y}} := \left(\|u\|^\frac{n}{m}_{W_0^{m,\frac{n}{m}}(\Omega)} + \|v\|^\frac{n}{m}_{W_0^{m,\frac{n}{m}}(\Omega)}\right)^{\frac{m}{n}}$$
where $\|u\|_{W_0^{m,\frac{n}{m}}(\Omega)}^{\frac{n}{m}} := \int_{\Omega} |\nabla^m u|^\frac{n}{m} dx.$\\
Recently, Megrez et al. \cite{MSK} 
proved the following Moser-Trudinger inequality in the product space for the case $n=2,\ m=1$ ({\it i.e. }$\mathcal{Y}=H^1_0(\Omega) \times H^1_0(\Omega))$. Precisely, they established - for $(u,v) \in \mathcal{Y}$ and $\Omega \subset \mathbb{R}^2$ a smooth bounded domain, $$\sup_{\|(u,v)\|_{\mathcal{Y}}=1}\int_{\Omega} \exp({\rho (u^2+v^2)}) dx < \infty ,\ \text{provided} \ \  \rho \leq 4 \pi.$$
\noindent In this article, we first establish the non-singular version of Moser-Trudinger and Adams-Moser-Trudinger inequalities in higher dimensional product spaces. Precisely, we prove the following new result:
\begin{thm}\label{MTST0}
For $(u,v) \in \mathcal{Y}$, $n,m\in \mathbb N$ such that $n\geq 2m$ and $\Omega \subset \mathbb{R}^n$ is a bounded domain, we have $$\int_{\Omega} \exp\left({\Theta\left(|u|^{\frac{n}{n-m}}+|v|^{\frac{n}{n-m}}\right)}\right) dx < \infty$$ for any $\Theta>0.$ Moreover,
\begin{equation}\label{main112}
\sup_{\|(u,v)\|_{\mathcal{Y}}=1}\int_{\Omega} \exp\left(\Theta \left(|u|^{\frac{n}{n-m}}+|v|^{\frac{n}{n-m}}\right)\right) dx < \infty ,\ \text{provided} \ \  \Theta \leq \frac{\zeta_{n,m}}{2_{n,m}}
\end{equation}
where $2_{n,m}= 2^{\frac{n-2m}{n-m}}.$ Furthermore if $\Theta > \frac{\zeta_{n,m}}{2_{n,m}}$, then there exists a pair $(u,v) \in \mathcal{Y}$ with $\|(u,v)\|_{\mathcal{Y}}=1$ such that the supremum in \eqref{main112} is infinite.
\end{thm}

\noindent As an consequence of Theorem \ref{MTST0}, we prove the following version of Lions' Lemma \cite{Lions} in the product space $\mathcal{Y}$.
\begin{thm}\label{MTST01}
Let $(u_k, v_k) \in \mathcal{Y}$  such that $\|(u_k,v_k)\|_{\mathcal{Y}}=1$ for all $k$ and $(u_k,v_k) \rightharpoonup (u,v) \not\equiv (0,0)$ weakly in $\mathcal{Y}.$ Then for all $p <  \displaystyle\frac{\zeta_{n,m}}{2_{n,m} (1- \|(u,v)\|^\frac{n}{m})^\frac{m}{n-m}}$,
\begin{equation*}
\sup_{k \in \mathbb{N}} \int_{\Omega} \exp\left(p \left(|u_k|^{\frac{n}{n-m}} + |v_k|^{\frac{n}{n-m}}\right)\right) dx < \infty.
\end{equation*}
\end{thm}

\noindent Next, we prove the singular version of Moser-Trudinger inequality in the Cartesian product of Sobolov spaces when $m=1$.
\begin{thm}\label{MTST00}
For $(u,v) \in \mathcal{Y}= W^{1,n}_0(\Omega)\times W^{1,n}_0(\Omega)$, $n\geq 2$, $\lambda \in [0,n)$ and $\Omega \subset \mathbb{R}^n$ is a smooth bounded domain, we have $$\int_{\Omega} \frac{\exp(\beta(|u|^{\frac{n}{n-1}}+|v|^{\frac{n}{n-1}}))}{|x|^\lambda} dx < \infty$$ for any $\beta>0.$ Moreover,
\begin{equation}
\sup_{\|(u,v)\|_{\mathcal{Y}}=1}\int_{\Omega} \frac{\exp(\beta (|u|^{\frac{n}{n-1}}+|v|^{\frac{n}{n-1}}))}{|x|^\lambda} dx < \infty \ \text{if and only if} \ \  \frac{ 2_n\beta}{\alpha_n} + \frac{\lambda}{n} \leq 1
\end{equation}
where $2_n :=2_{n,1}=2^{\frac{n-2}{n-1}}$.
\end{thm}
\noindent Similarly we can prove  singular and non-singular  Moser-Trudinger inequalities in the  product space $\mathcal{Z}:= W^{1,n}(\Omega) \times W^{1,n}(\Omega)$ where $\Omega \subset \mathbb R^n$ is a bounded domain endowed with the norm $$\|(u,v)\|_{\mathcal{Z}} := \left(\|u\|^n_{W^{1,n}(\Omega)} + \|v\|^n_{W^{1,n}(\Omega)}\right)^{\frac{1}{n}}$$
where $\|u\|_{W^{1,n}(\Omega)}^{n} := \displaystyle\int_{\Omega} (|u|^n + |\nabla u|^n )~dx$. Precisely we establish the following new result.
\begin{thm}\label{MTST11}
For $(u,v) \in \mathcal{Z}, n \geq 2$, $\lambda \in [0,n)$ and $\Omega \subset \mathbb{R}^n$ be a smooth bounded domain, we have $$\int_{\Omega} \frac{\exp(\tilde{\beta}(|u|^{\frac{n}{n-1}}+|v|^{\frac{n}{n-1}}))}{|x|^\lambda} dx < \infty$$ for any $\tilde{\beta}>0.$ Moreover,
\begin{equation}\label{1main1}
\sup_{\|(u,v)\|_{\mathcal{Z}}=1}\int_{\Omega} \frac{\exp(\tilde{\beta}(|u|^{\frac{n}{n-1}}+|v|^{\frac{n}{n-1}}))}{|x|^\lambda} dx < \infty \ \text{ if and only if} \ \  \frac{2 \tilde{\beta}}{\alpha_n} + \frac{\lambda}{n} \leq 1.
\end{equation}
\end{thm}
\vspace{.2cm}
\noindent As an application of Theorems \ref{MTST0} and \ref{MTST01}, we study the existence of solution for the following Kirchhoff system involving the exponential nonlinearity of Choquard type
\begin{equation*}
   (KCS)  \left\{
         \begin{alignedat}{2}
             {} -m\left(\int_\Omega|\nabla u|^ndx\right)\Delta_n u
             & {}=  \left(\int_{\Omega} \frac{F(y,u,v)}{|x-y|^\mu}dy\right)f_1(x,u,v),\; u>0
             && \quad\mbox{ in }\, \Omega ,
             \\
             -m\left(\int_\Omega|\nabla v|^ndx\right)\Delta_n v
             & {}=  \left(\int_{\Omega} \frac{F(y,u,v)}{|x-y|^\mu}dy\right)f_2(x,u,v),\; v>0
             && \quad\mbox{ in }\, \Omega ,
             \\
             u,v & {}= 0
             && \quad\mbox{ on }\, \partial\Omega
          \end{alignedat}
     \right.
\end{equation*}
where $\Omega$ is a smooth bounded domain, $n \geq 2$, $0 < \mu<n$, $m: \mathbb{R}^+ \to \mathbb{R}^+$ is a continuous function, $\Delta_n u := \text{div}( |\nabla u|^{n-2} \nabla u)$, $F$  satisfies suitable growth assumptions and $f_1=\frac{\partial F}{\partial u},\; f_2=\frac{\partial F}{\partial v}$. The system of type $(KCS)$ having doubly nonlocal feature was not studied in the literature so far.\\
In 1883, Kirchhoff extends the classical D'Alembert wave equation to the following model:
$$u_{tt} - m\left(\int_{0}^L \left|u_x \right|^2\right) u_{xx} =0 $$
for $t \geq 0$ and $0<x <L$,
where $u(t,x)$ is the lateral displacement at the space coordinate $x$ and time $t$, $m(t)= \frac{p_0}{\rho h} + \frac{Y_m }{2 \rho  L} t ,Y_m$  is Young modulus, $\rho$ is mass density, $h$ is the cross section area, $L$ is the length of string, $p_0$ the initial axial tension. In the case of degenerate Kirchhoff problems  $m(0)=0$ {\it i.e} initial axial tension is zero. From the physical point of view, $m(0)=0$ can be interpreted as follows: The base tension of the string is zero and $m$ measures the change of the tension in the string caused by the change of its length during vibration. It shows that the presence of nonlinear coefficient $m$ is meaningful to be considered. We cite \cite{Alves, Alves1, AGMS, LiLi} and there references within  for further considerations.\vspace{0.2cm} \\
On an analogous note, the non-local problems involving  the following convolution type non-linearity
$$ -\Delta u + V(x) u = \left(|x|^{-\mu} * F(x,u) \right) f(x,u) \ \ \text{in} \ \ \mathbb{R}^n, \ \ \ \mu \in (0,n)$$ got attention by a large scale of researchers due to its occurrence in several physical models. In 1954, Pekar \cite{Pekar} used such equation to describe the quantum theory of a polaron at rest. In 1976, Choquard \cite{Lieb} used it to model an electron trapped in its own hole. In 2000, Berg\'e and Couairon \cite{BC} studied standing waves of the non-linear non-local Schr\"odinger equation which influence the propagation of electromagnetic waves in Plasma. These kind of non-linearities also play a crucial role in the Bose-Einstein condensation \cite{DGPS}. For interested readers, we refer the recent survey paper on Choquard equations  by Moroz and Schaftingen \cite{ms-survey} and  Mukherjee and Sreenadh  \cite{MuSr}. In 2014, L\"u \cite{Lu} studied the non-degenerate Choquard equation with Kirchhoff operator in $\mathbb{R}^3$ and using the method of Nehari manifold established the existence of ground state solution. In \cite{PXZ}, authors studied the existence of non-negative solutions of a Schr\"odinger-Choquard-Kirchoff type $p$-fractional equation via variational methods. The problem of the type $(KCS)$ for the single equation, $n=2$ and without Choquard non-linearity, {\it i.e.}
$$-m\left(\int_{\Omega} |\nabla u|^2\right) \Delta u = f(x,u) \ \ \text{in}  \ \Omega, \ \ u=0 \ \ \text{in} \  \partial\Omega,$$ was studied by Figueiredo and Severo \cite{FS} which was generalized for $n$-Laplace equation by Goyal et al \cite{GMS}.
Recently in \cite{AGMS, AGMS1}, authors have studied the Kirchhoff-Choquard problem with exponential nonlinearity in the case of a single equation and proved the existence of solution using variational methods. \vspace{.3cm} \\
Coming to the system of equations, Megrez et al. \cite{MSK} studied the following parametric gradient system with exponential nonlinearity
\begin{equation*}
     \left\{
         \begin{alignedat}{2}
             {} - \Delta u
             & {}=  \lambda u^q +f_1(u,v),\; u>0
             && \quad\mbox{ in }\, \Omega ,
             \\
            - \Delta v
             & {}=  \lambda v^q +f_2(u,v),\; u>0
             && \quad\mbox{ in }\, \Omega ,
             \\
             u,v & {}= 0
             && \quad\mbox{ on }\, \partial\Omega,
          \end{alignedat}
     \right.
\end{equation*}
where $\Omega \subset \mathbb{R}^2$ is a smooth bounded domain, $q\in (0,1)$ and proved the existence and non-existence result for a suitable range of $\lambda$ by using generalized version of mountain-pass lemma. Motivated from the above articles, we study the Kirchhoff system of equations $(KCS)$ with exponential nonlinearity of Choquard type. To the best of our knowledge there is no work available till date, for Kirchhoff system involving exponential non-linearity of Choquard type even for n=2 and $m(t)\equiv1$. So, in this regard our work is first of its kind. Also on an important note, we work with the nonlinear $n$-Laplace operator for $n \geq 2.$ \vspace{.2cm}.\\
We recall the well known Hardy-Littlewood-Sobolev inequality:
 \begin{pro}\label{HLS}
Let $t,r>1$ and $0<\mu<n $ with $1/t+\mu/n+1/r=2$, $f \in L^t(\mathbb{R}^n)$ and $h \in L^r(\mathbb{R}^n)$. Then there exists a sharp constant $C(t,n,\mu,r)>0$, independent of $f,h$ such that
 \begin{equation}\label{HLSineq}
 \int_{\mathbb{R}^n}\int_{\mathbb{R}^n} \frac{f(x)h(y)}{|x-y|^{\mu}}\mathrm{d}x\mathrm{d}y \leq C(t,n,\mu,r)\|f\|_{L^t(\mathbb{R}^n)}\|h\|_{L^r(\mathbb{R}^n)}.
 \end{equation}
If $t =r = \textstyle\frac{2n}{2n-\mu}$ then
 \[C(t,n,\mu,r)= C(n,\mu)= \pi^{\frac{\mu}{2}} \frac{\Gamma\left(\frac{n}{2}-\frac{\mu}{2}\right)}{\Gamma\left(n-\frac{\mu}{2}\right)} \left\{ \frac{\Gamma\left(\frac{n}{2}\right)}{\Gamma(n)} \right\}^{-1+\frac{\mu}{n}}.  \]
 In this case there is equality in \eqref{HLSineq} if and only if $f\equiv (constant)\ h$ and
 \[h(x)= A(\gamma^2+ |x-a|^2)^{\frac{-(2n-\mu)}{2}}\]
 for some $A \in \mathbb C$, $0 \neq \gamma \in \mathbb{R}$ and $a \in \mathbb{R}^n$.
 \end{pro}
\noindent Now we state the assumptions on $m$ and $f$ for the problem $(KCS)$. Let $m: \mathbb{R}^+ \to \mathbb{R}^+$ be a continuous function satisfying the following conditions:
\begin{enumerate}
\item[(m1)] $M(t+s) \geq M(t) + M(s)$ for all $t,s \geq 0$ where $M(t)$ is the primitive of the function $m$.
\item[(m2)] There exist constants $c_0, c_1, c_2 >0$ and $ \tilde{t} >0$ such that for some $r,z \in \mathbb{R}^+$
$$ m(t) \geq c_0   \ \ \text{or}\ \  m(t) \geq t^z, \; \text{for all} \ t \geq 0$$ and
$$\ m(t) \leq c_1 + c_2 t^r, \ \text{for all} \ t \geq \tilde{t}.$$
\item[(m3)] The map $t \mapsto \frac{m(t)}{t}$ is non-increasing for $t >0.$
\end{enumerate}
We remark that the assumption $(m2)$ covers both degenerate as well as non-degenerate case of the Kirchhoff term. \vspace{.2cm} \\
\textbf{Example 1:} An example of a function $m$ satisfying $(m1)-(m3)$ is $m(t)= d_0 + d_1 t^{\beta}$ for $\beta<1$ and $d_0,d_1\geq 0$. \vspace{.2cm}
\\
\noindent Let the function $F: \Omega \times \mathbb{R}^2 \to \mathbb{R}$ be continuously differentiable with respect to second and third variable and  of the form $F(x,t,s)=  h(x,t,s) \exp(|t|^{\frac{n}{n-1}}+ |s|^{\frac{n}{n-1}})$ such that
$$ f_1(x,t,s) :=\frac{\partial F}{\partial t}(x,t,s) = h_1(x,t,s) \exp(|t|^{\frac{n}{n-1}}+ |s|^{\frac{n}{n-1}}),$$
$$f_2(x,t,s):= \frac{\partial F}{\partial s}(x,t,s) = h_2(x,t,s) \exp(|t|^{\frac{n}{n-1}}+ |s|^{\frac{n}{n-1}}).$$ 
We assume $h_i$'s for $i=1,2$ are continuous functions satisfying the following conditions-
\begin{enumerate}
\item[(f1)] $h_i(x,t,s)=0$ when either $t \leq 0$ or $s \leq 0$ and $h_i(x,t,s) >0$ when $t,s > 0$, for all $x \in \Omega$ and $i=1,2$.
\item[(f2)] For any $\epsilon >0$ and $i=1,2$
$$\lim_{t,s \to \infty} \sup_{x \in \overline{\Omega}} h_i(x,t,s) \exp(-\epsilon(|t|^{\frac{n}{n-1}}+ |s|^{\frac{n}{n-1}})) = 0,$$
$$\lim_{t,s \to \infty} \inf_{x \in \overline{\Omega}} h_i(x,t,s) \exp(\epsilon(|t|^{\frac{n}{n-1}}+ |s|^{\frac{n}{n-1}})) = \infty.$$
\item[(f3)] There exists 
 \begin{equation*}
l > \left\{
\begin{split}
& \max\left\{n-1, \frac{n(r+1)}{2}\right\} \; &&  \text{when}\ m \ \text{is non-degenerate},\\
&  \max\left\{n-1,\frac{n(z+1)}{2}, \frac{n(r+1)}{2}\right\} \; &&  \text{when}\ m \ \text{is degenerate.}
\end{split}
\right.
\end{equation*}
such that the maps 
$t \mapsto \frac{f_1(x,t,s)}{|t|^{l}}$, $s \mapsto \frac{f_2(x,t,s)}{|s|^{l}}$ are increasing functions of $t$ (uniformly in $s$ and $x$) and $s$ (uniformly in $t$ and $x$) respectively.
\item[(f4)] There exist $q, s_0,t_0, M_0 >0$ such that $s^q F(x,t,s) \leq M_0 f_2(x,t,s)$ for all $s \geq s_0$ and $t^q F(x,t,s) \leq M_0 f_1(x,t,s)$ for all $t \geq t_0$ uniformly in $x \in \Omega.$
\item[(f5)] There exists a $\gamma$ satisfying $\frac{n-2}{2} < \gamma$ such that  $\lim\limits_{(t,s) \to (0,0)} \frac{f_i(x,t,s)}{s^\gamma + t^{\gamma}} = 0$ holds for $i=1,2$.
\end{enumerate}
\begin{thm}\label{existence}
Let $m$ satisfies $(m1)-(m3)$ and $f$ satisfies $(f1)-(f5)$ and 
\begin{equation*}
\lim_{t,s\to \infty} \frac{(f_1(x,t,s)t + f_2(x,t,s)s) F(x,t,s)}{\exp( q (|t|^{\frac{n}{n-1}} + |s|^{\frac{n}{n-1}}))} = \infty \ \text{uniformly in} \ x \in \overline{\Omega}.
\end{equation*}
for some $q>2$. Then there exists a positive weak solution of the problem $(KCS)$.
\end{thm}
Turning to the layout of the paper: In section 2, we prove the Theorems \ref{MTST0}, \ref{MTST01}, \ref{MTST00}, \ref{MTST11}. In section 3, as an application of Theorem \ref{MTST0} ,\ref{MTST01}, \ref{MTST11}, we prove the main existence result: Theorem \ref{existence} for the system of equations $(KCS)$.

\section{Proof of the main results}
\begin{Lem}\label{basic}
If $a, b>0$ such that $a+b=1$ then $a^{\alpha} + b^{\alpha} \leq 2^{1-\alpha}$ for all $0< \alpha <1.$
\end{Lem}
\begin{proof}
Let $r: (0,1] \times (0,1] \to \mathbb{R}$ such that $r(a,b)= a^{\alpha} +b^{\alpha}$ and $a+b=1$ then $$\widetilde{r}(a):=r(a, 1-a)= a^{\alpha} + (1-a)^{\alpha}$$ and $$\frac{d}{da}\widetilde{r}(a)= \alpha ( a^{\alpha-1} - (1-a)^{\alpha-1}) =0$$ gives $a=\frac{1}{2}$, which is the point of maximum (since $\frac{d}{da}\left(\frac{d}{da} \tilde{r}\right)(a)\big|_{a=\frac{1}{2}} < 0$ ). Therefore the maximum value of $\widetilde{r}$ in $(0,1]$ is $2^{1-\alpha}.$
\end{proof}
\noindent \textbf{Proof of Theorem \ref{MTST0}:}\\
We denote $\|\cdot\|:= \|\cdot\|_{W_0^{m,\frac{n}{m}}(\Omega)}$. Without loss of generality, let $(u,v) \in \mathcal{Y}\setminus\{(0,0)\}$ be such that $\|(u,v)\|_\mathcal{Y} =1$. If either $u\equiv 0$ or $v \equiv 0$, the result follows from Theorem \ref{TM-ineq}. \\
We set $\alpha = \frac{m}{n-m}$, $a=\|u\|^\frac{n}{m}$ and $b=\|v\|^\frac{n}{m}$ then Lemma \ref{basic} gives us that $$\frac{\|u\|^{\frac{n}{n-m}}}{2_{n,m}} + \frac{\|v\|^{\frac{n}{n-m}}}{2_{n,m}} \leq 1$$
where $2_{n,m}= 2^{\frac{n-2m}{n-m}}.$\\
\textbf{Case 1:} Let $\frac{\|u\|^{\frac{n}{n-m}}}{2_{n,m}} + \frac{\|v\|^{\frac{n}{n-m}}}{2_{n,m}} < 1.$ \\
\noindent Then there exists $1 < c := c(u,v) < \infty $ such that
$$\frac{\|u\|^{\frac{n}{n-m}}}{2_{n,m}} + \frac{\|v\|^{\frac{n}{n-m}}}{2_{n,m}} + \frac{1}{c}= 1.$$
Using the generalized H\"older's inequality and $\Theta \leq \frac{\zeta_{n,m}}{2_{n,m}}$ we obtain
\begin{equation}\label{est1}
\begin{split}
\int_{\Omega} & \exp({\Theta (|u|^{\frac{n}{n-m}} + |v|^{\frac{n}{n-m}})})\\
&  \leq |\Omega|^{\frac{1}{c}} \left(\int_{\Omega} \exp({\Theta 2_{n,m} \left(\frac{|u|}{\|u\|}\right)^{\frac{n}{n-m}}})\right)^{\frac{\|u\|^{\frac{n}{n-m}}}{2_{n,m}}} \left(\int_{\Omega} \exp({\Theta 2_{n,m} \left(\frac{|v|}{\|v\|}\right)^{\frac{n}{n-m}}})\right)^{\frac{\|v\|^{\frac{n}{n-m}}}{2_{n,m}}}\\
&\leq C \left(\int_{\Omega} \exp({\zeta_{n,m} \left(\frac{|u|}{\|u\|}\right)^{\frac{n}{n-m}}})\right)^{\frac{\|u\|^{\frac{n}{n-m}}}{2_{n,m}}}  \left(\int_{\Omega} \exp({\zeta_{n,m}\left(\frac{|v|}{\|v\|}\right)^{\frac{n}{n-m}}})\right)^{\frac{\|v\|^{\frac{n}{n-m}}}{2_{n,m}}}
\end{split}
\end{equation}
where $C$ is a positive constant depending on $|\Omega|$ but independent of $u,v.$ \vspace{.2cm}\\
\textbf{Case 2:} $\frac{\|u\|^{\frac{n}{n-m}}}{2_{n,m}} + \frac{\|v\|^{\frac{n}{n-m}}}{2_{n,m}} = 1.$ \vspace{.2cm}\\
Applying the H\"older's inequality and $\Theta\leq \frac{\zeta_{n,m}}{2_{n,m}}$ we obtain
\begin{equation}\label{est2}
\begin{split}
\int_{\Omega} & \exp({\Theta (|u|^{\frac{n}{n-m}} + |v|^{\frac{n}{n-m}})}) \\
& \leq  \left(\int_{\Omega} \exp({\Theta 2_{n,m} \left(\frac{|u|}{\|u\|}\right)^{\frac{n}{n-m}}})\right)^{\frac{\|u\|^{\frac{n}{n-m}}}{2_{n,m}}}  \left(\int_{\Omega} \exp({\Theta 2_{n,m} \left(\frac{|v|}{\|v\|}\right)^{\frac{n}{n-m}}})\right)^{\frac{\|v\|^{\frac{n}{n-m}}}{2_{n,m}}}\\
&\leq
 \left(\int_{\Omega} \exp({\zeta_{n,m} \left(\frac{|u|}{\|u\|}\right)^{\frac{n}{n-m}}})\right)^{\frac{\|u\|^{\frac{n}{n-m}}}{2_{n,m}}}  \left(\int_{\Omega} \exp({\zeta_{n,m}\left(\frac{|v|}{\|v\|}\right)^{\frac{n}{n-m}}})\right)^{\frac{\|v\|^{\frac{n}{n-m}}}{2_{n,m}}}.
\end{split}
\end{equation}
Now by combining \eqref{est1}, \eqref{est2} and taking supremum over $\|(u,v)\|_{\mathcal{Y}} =1$, we obtain the desired inequality \eqref{main112}.
For the remaining part of the proof, we assume that $0 \in \Omega$ and seek use of the Adams function to construct a sequence of test functions. Let us denote $\mathcal{B}(0,l)\eqdef \mathcal{B}_l$ as a ball with center $0$ and radius $l$ in $\mathbb{R}^n$ then without loss of generality, we can assume that $ B(0,l) \subset \Omega$ for $ \in (0,1)$. We recall the following result (see \cite{Lakkis}): For $l\in (0,1)$, there exists
\begin{equation}\label{ul}
U_l \in \{u \in W_0^{m,\frac{n}{m}}(\Omega): u|_{\mathcal{B}_l}=1 \}
\end{equation} such that $$\|U_l\|^{\frac{n}{m}}= C_{m ,\frac{n}{m}}(\mathcal{B}_l; \mathcal{B}_1) \leq \left(\frac{ \zeta_{n,m}}{n \log\left(\frac{1}{l}\right)}\right)^{\frac{n-m}{m}}$$ where $C_{m,\frac{n}{m}}(K,E)$ is the conductor capacity of $K$ in $E$ whenever $E$ is an open set and $K$ is relatively compact subset of $E$ and $C_{m ,\frac{n}{m}}(K;E) \eqdef \inf\{\|u\|^{\frac{n}{m}}: u \in C_0^\infty(E), u|_{K}=1\}.$
Let us set $\sigma>0$ and $l = \frac{1}{k},$ for $k \in \mathbb N$. Also we define
\begin{equation*}
A_k(x)=\left\{
\begin{split}
&\left(\frac{n \log(k)}{ \zeta_{n,m}}\right)^\frac{n-m}{n} U_{\frac{1}{k}}\left(\frac{x}{\sigma}\right) \  \text{if} \ |x| < \sigma;\\
& 0 \ \ \ \ \quad \quad \quad \ \ \ \ \  \ \ \ \ \ \ \ \ \ \ \ \ \ \ \  \text{if} \ |x| \geq  \sigma.\\
\end{split}
\right.
\end{equation*}
Then we have $A_k(x) \big|_{\mathbb{B}_{\frac{\sigma}{k}}}=\left(\frac{n \log(k)}{ \zeta_{n,m}}\right)^\frac{n-m}{n}$ and $\|A_k\| \leq 1$, Now we consider $$Z_k= c_1 w_k \ \ \ \text{and}\ \ V_k= c_2 w_k$$
where $w_k(x)= \frac{A_k}{\|A_k\|}$ and  $c_1, c_2 \in \mathbb{R}^+$  verifies $$c_1^\frac{n}{m} + c_2^\frac{n}{m} =1 \ \text{and}\  c_1^{\frac{n}{n-m}} + c_2^{\frac{n}{n-m}} = 2_{n,m}$$
which implies that supp$(w_k) \subset B_{\sigma}(0)$ and $\|w_k\|=1$ for all $k$. The existence of $c_1, c_2$ can be proved using Lemma \ref{basic}. Thus we obtain
\begin{equation*}
\begin{split}
\|Z_k, V_k\|_{\mathcal{Y}} &= \left(\|Z_k\|^\frac{n}{m} + \|V_k\|^\frac{n}{m} \right)^{\frac{m}{n}} = \left(c_1^\frac{n}{m} \|w_k\|^\frac{n}{m} + c_2^\frac{n}{m} \|w_k\|^ \frac{n}{m}\right)^{\frac{m}{n}}\\
&= \|w_k\| (c_1^\frac{n}{m} + c_2^\frac{n}{m})^{\frac{m}{n}}=1.
\end{split}
\end{equation*}
So if $\Theta > \frac{\zeta_{n,m}}{2_{n,m}}$, then for some $\epsilon>0$, $\Theta= (1+\epsilon)\frac{\zeta_{n,m}}{2_{n,m}}$ which gives that
\begin{equation*}
\begin{split}
\int_{\Omega} \exp({\Theta (|U_k|^{\frac{n}{n-m}} + |V_k|^{\frac{n}{n-m}})} )& \geq  \int_{B_{\frac{\sigma}{k}}} \exp\left({(1+\epsilon)\frac{\zeta_{n,m}}{2_{n,m}}(|w_k|^{\frac{n}{n-m}} (c_1^{\frac{n}{n-m}} + c_2^{\frac{n}{n-m}}) )}\right)\\
& = \int_{B_{\frac{\sigma}{k}}} k^{n(1+\epsilon)} \geq C_3 k^{\epsilon}  \to \infty \ \ \text{as} \ k \to \infty.
\end{split}
\end{equation*}
This completes the proof. \qed\\

\noindent \textbf{Proof of Theorem \ref{MTST01} :}
Using Brezis-Lieb lemma, it is easy to see that
$$\lim_{k \to \infty} \|(u_k-u), (v_k-v)\|^\frac{n}{m}_\mathcal{Y} = 1- \|(u,v)\|_{\mathcal{Y}}^\frac{n}{m}$$ and
\begin{equation}\label{est5}
|u_k|^\frac{n}{n-m} \leq \left(|u_k-u|^\frac{n}{n-m} + |u|^\frac{n}{n-m}\right)+ C(|u_k-u|^{\frac{m}{n-m}}|u|+|u|^{\frac{m}{n-m}}|u_k-u| )
\end{equation}
where $C\eqdef C(n,m)>0$. Now for any $\epsilon>0$, from Young's inequality we have that 
\[ab \leq \frac{m}{n}(\epsilon a)^{\frac{n}{m}}+ \frac{n-m}{n}\left(\frac{b}{\epsilon}\right)^{\frac{n}{n-m}}.\]
This gives us 
\begin{align*}
|u_k|^\frac{n}{n-m} & \leq \left((1+C_1\epsilon^{\frac{n}{m}}+C_1{\epsilon^{\frac{n}{n-m}}})|u_k-u|^\frac{n}{n-m} +  (1+C_1\epsilon^{\frac{-n}{m}}+C_1{\epsilon^{\frac{-n}{n-m}}})|u|^\frac{n}{n-m}\right)\\
&:= C_{1,\epsilon}|u_k-u|^\frac{n}{n-m} + C_{1,\epsilon}^\prime|u|^\frac{n}{n-m}\ \text{(say)}. 
\end{align*}
Similarly we also have 
\[|v_k|^\frac{n}{n-m} \leq  C_{1,\epsilon}|v_k-v|^\frac{n}{n-m} + C_{1,\epsilon}^\prime|v|^\frac{n}{n-m}.\]
Therefore by using  H\"older inequality and above estimates we obtain,
\begin{equation*}
\begin{split}
&\int_{\Omega} \exp\left({p (|u_k|^{\frac{n}{n-m}} + |v_k|^{\frac{n}{n-m}})}\right) dx \leq \left(\int_{\Omega} \exp\left({p C_{1,\epsilon}r_1 \left(|u_k-u|^\frac{n}{n-m} +|v_k-v|^\frac{n}{n-m}\right)}\right) dx \right)^{\frac{1}{r_1}} \\
& \ \ \ \ \ \ \ \ \ \ \ \ \ \ \ \ \ \ \ \  \ \ \ \ \quad \quad \quad \quad \quad \quad \quad \quad \quad \quad . \left(\int_{\Omega} \exp\left({p C_{1,\epsilon}^\prime r_2  \left(|u|^\frac{n}{n-m} +|v|^\frac{n}{n-m}\right)}\right) dx \right)^{\frac{1}{r_2}}\\
& \leq C(n,m,u,v,r_2)\left(\int_{\Omega} \exp\left({pC_{1,\epsilon} r_1 (\|(u_k-u), (v_k-v)\|_{\mathcal{Y}})^\frac{n}{n-m}}\right.\right.\\
& \ \ \ \ \left.\left.\left(\left(\frac{|u_k-u|}{\|(u_k-u), (v_k-v)\|_{\mathcal{Y}}}\right)^\frac{n}{n-m} + \left(\frac{|v_k-v|}{\|(u_k-u), (v_k-v)\|_{\mathcal{Y}}}\right)^\frac{n}{n-m}\right)\right) dx \right)^{\frac{1}{r_1}} 
\end{split}
\end{equation*}
where $r_1$ and $r_2$ are H\"{o}lder conjugate to each other and $C(n,m,u,v,r_2)$ is a positive constant independent of $k$. Now since $C_{1,\epsilon} \to 1$ as $\epsilon \to 0$,  by choosing $\epsilon>0$ small enough and $r_1 >1$ very close to 1 such that 
\[pr_1C_{1,\epsilon}(1-\|(u,v)\|_{\mathcal{Y}}^{\frac{n}{m}})^{\frac{m}{n-m}}< \frac{\zeta_{n,m}}{2_{n,m}}\]
 we get the desired result, by using Theorem \ref{MTST0}.
\qed \vspace{0.2cm} \\
To prove the following Singular Moser-Trudinger inequality in cartesian product of Sobolev space taking $m=1$ and using the idea of Theorem $2.1$ in \cite{AdiSan}. \vspace{.2cm}\\
\noindent \textbf{Proof of Theorem \ref{MTST00}:} \vspace{.2cm}\\
We denote $\|\cdot\|:= \|\cdot\|_{W_0^{1,n}(\Omega)}$ in this proof. Let $(u,v) \in \mathcal{Y}$  be such that $\|(u,v)\|_\mathcal{Y} =1$, $\lambda \in (0,n)$ and $\beta>0$. Then following two cases arise: \vspace{.1cm}\\
\textbf{Case 1:} Let $\frac{\beta 2_{n}}{\alpha_{n}} + \frac{\lambda}{n} < 1$ then we choose $t>1$ such that  $$\frac{\beta 2_{n}}{\alpha_n} + \frac{\lambda t}{n} = 1.$$
Now by using Cauchy-Schwarz  inequality and Theorem \ref{MTST0}, we obtain
\begin{equation}\label{est01}
\begin{split}
\int_{\Omega} \frac{\exp({\beta (|u|^{\frac{n}{n-1}} + |v|^{\frac{n}{n-1}})})}{|x|^\lambda} & \leq  \left(\int_{\Omega} \exp\left({\frac{\alpha_{n}}{2_{n}} \left(|u|^{\frac{n}{n-1}} + |v|^{\frac{n}{n-1}}\right)}\right)\right)^{\frac{\beta 2_{n}}{\alpha_{n}}}. \left(\int_{\Omega} \frac{1}{|x|^\frac{n}{t}}\right)^{\frac{\lambda t}{n}} \leq C
\end{split}
\end{equation}
where $C$ is a constant independent of $u,v.$ \vspace{.2cm}\\
\textbf{Case 2:} Let $\frac{\beta 2_{n}}{\alpha_{n}} + \frac{\lambda}{n} = 1.$ Then
from standard symmetrization and density arguments we can reduce to the case $\Omega$ being a ball $B(0,R)$ with centre origin and radius $R$ and $u,v$ being positive smooth and radial functions. Therefore 
\begin{equation}\label{est02}
\int_{B(0,R)} (|\nabla u|^n + |\nabla v|^n) dx = \omega_{n-1} \int_{0}^R ((u'(r))^n + (v'(r))^n) r^{n-1} dr
\end{equation}
and
\begin{equation}\label{est03}
\int_{B(0,R)} \frac{\exp\left(\frac{s \alpha_{n}}{2_{n}}(|u|^{\frac{n}{n-1}} + |v|^{\frac{n}{n-1}})\right)}{|x|^{(1-s) n}} dx= \int_{0}^R \exp\left(\frac{s \alpha_{n}}{2_{n}}(|u|^{\frac{n}{n-1}} + |v|^{\frac{n}{n-1}})\right) r^{sn-1} dr
\end{equation}
where $s= \frac{\beta 2_{n}}{\alpha_{n}}$ so that $\lambda= (1-s) n.$ Now we set
$$\tilde{u}(r)= s^{\frac{n-1}{n}} u(r^\frac{1}{s}) \ \text{and} \  \tilde{v}(r)= s^{\frac{n-1}{n}} v (r^\frac{1}{s})\ \text{for} \ r \in [0, R^s].$$
Therefore
\begin{equation}\label{est04}
\int_{0}^R ((u'(r))^n + (v'(r))^n) r^{n-1} dr= \int_{0}^{R^s} ((\tilde{u}'(r))^n + (\tilde{v}'(r))^n) r^{n-1} dr,
\end{equation}
\begin{equation}\label{est05}
\int_{0}^R \exp\left(\frac{s \alpha_{n}}{2_{n}}(|u|^{\frac{n}{n-1}} + |v|^{\frac{n}{n-1}})\right) r^{sn-1} dr= \frac{1}{s} \int_{0}^{R^s} \exp\left(\frac{\alpha_{n}}{2_{n}}(|\tilde{u}|^{\frac{n}{n-1}} + |\tilde{v}|^{\frac{n}{n-1}})\right) r^{n-1} dr.
\end{equation}
Now by combining \eqref{est02}-\eqref{est05} and taking supremum, we obtain 
\begin{align*}
&\sup_{\|(u,v)\|_{\mathcal{Y}}=1} \int_{B(0,R)} \frac{\exp\left(\frac{s \alpha_{n}}{2_{n}}(|u|^{\frac{n}{n-1}} + |v|^{\frac{n}{n-1}})\right)}{|x|^{(1-s) n}} dx \\
&\leq \sup_{\|(\tilde u, \tilde v)\|_{\mathcal{Y}}=1} \frac{R^{s(n-1)}}{s} \int_{0}^{R^s} \exp\left(\frac{\alpha_{n}}{2_{n}}(|\tilde{u}|^{\frac{n}{n-1}} + |\tilde{v}|^{\frac{n}{n-1}})\right)  dr <\infty
\end{align*}
which is the desired inequality.
For the    remaining part of the proof, we assume $0 \in \Omega$ and define
\begin{equation*}
w_k(x)=\frac{1}{\omega_{n-1}^{\frac{1}{n}}}\left\{
\begin{split}
& (\log k)^{\frac{n-1}{n}},\; 0\leq |x|\leq \frac{\rho}{k},\\
& \frac{\log \left(\frac{\rho}{|x|}\right)}{(\log k)^{\frac{1}{n}}}, \; \frac{\rho}{k}\leq |x|\leq \rho,\\
& 0,\;\ \ \ \ \ \ \ \ \ \ \  |x|\geq \rho
\end{split}
\right.
\end{equation*}
such that $supp(w_k) \subset B_{\rho}(0)$ and $\|w_k\|=1$ for all $k.$
Let $c_1, c_2 \in \mathbb{R}^+$  such that $c_1^n + c_2^n =1$ and $c_1^{\frac{n}{n-1}} + c_2^{\frac{n}{n-1}} = 2^{\frac{n-2}{n-1}}$ (The existence of $c_1, c_2$ can be proved by taking the maximum of function mentioned in Lemma \ref{basic}).\\
Also we define $$U_k= c_1 w_k \ \ \ \text{and}\ \ V_k= c_2 w_k$$
such that
\begin{equation*}
\begin{split}
\|U_k, V_k\|_{\mathcal{Y}} = \left(\|U_k\|^n + \|V_k\|^n \right)^{\frac{1}{n}} = \left(c_1^n \|w_k\|^n + c_2^n \|w_k\|^n\right)^{\frac{1}{n}}= \|w_k\| (c_1^n + c_2^n)^{\frac{1}{n}}=1.
\end{split}
\end{equation*}
Now let $\beta > \left(1-\frac{\lambda}{n}\right) \frac{\alpha_n}{2_n}$, then for some $\epsilon>0$, $\beta= (1+\epsilon) \left(1-\frac{\lambda}{n}\right)\frac{\alpha_n}{2_n}$ and
\begin{equation*}
\begin{split}
&\int_{\Omega} \frac{\exp\left({\beta (|U_k|^{\frac{n}{n-1}} + |V_k|^{\frac{n}{n-1}})}\right)}{|x|^\lambda} \geq  \int_{B_{\frac{\rho}{k}}} \frac{\exp\left({(1+\epsilon) \left(1-\frac{\lambda}{n}\right)\frac{\alpha_n}{2_n}|w_k|^{\frac{n}{n-1}} (c_1^{\frac{n}{n-1}} + c_2^{\frac{n}{n-1}})}\right)}{|x|^\lambda}\\
&\geq  \int_{B_{\frac{\rho}{k}}} k^{n(1+\epsilon) \left(1-\frac{\lambda}{n}\right)+ \lambda} \geq C_3 k^{\epsilon(n-\lambda)}  \to \infty \ \ \text{as} \ k \to \infty.
\end{split}
\end{equation*}
\qed\\
\textbf{Proof of Theorem \ref{MTST11}:} The proof can be done by following the same steps as in Theorems \ref{MTST0} and \ref{MTST00}.\qed

\section{Applications}
In this section, we study the following system of $n$-Kirchhoff Choquard equations with exponential nonlinearity
\begin{equation*}
     (KCS)
     \left\{
         \begin{alignedat}{2}
             {} -m(\int_\Omega|\nabla u|^ndx)\Delta_n u
             & {}=  \left(\int_{\Omega} \frac{F(y,u,v)}{|x-y|^\mu}dy\right)f_1(x,u,v),\; u>0
             && \quad\mbox{ in }\, \Omega ,
             \\
             -m(\int_\Omega|\nabla v|^ndx)\Delta_n v
             & {}=  \left(\int_{\Omega} \frac{F(y,u,v)}{|x-y|^\mu}dy\right)f_2(x,u,v),\; v>0
             && \quad\mbox{ in }\, \Omega ,
             \\
             u,v & {}= 0
             && \quad\mbox{ on }\, \partial\Omega,
          \end{alignedat}
     \right.
\end{equation*}
where $\Omega$ is a smooth bounded domain in $\mathbb R^n$, $0 < \mu<n$ and $m,f$ satisfies  assumptions $(m1)-(m3)$ and $(f1)-(f5)$. Let $\mathcal{P}:= W_0^{1,n}(\Omega) \times W_0^{1,n}(\Omega)$ endowed with the graph norm
$$\|(u,v)\| := \left(\|u\|^n_{W_0^{1,n}(\Omega)} + \|v\|^n_{W_0^{1,n}(\Omega)}\right)^{\frac{1}{n}}$$
where $\|u\|_{W_0^{1,n}(\Omega)}^{n} := \int_{\Omega} |\nabla u|^n dx.$ The study of the elliptic system $(KCS)$ is motivated by  Theorems \ref{MTST0} and \ref{MTST01}.
 Following is the notion of weak solution for $(KCS)$.
\begin{define}
A function $(u,v) \in \mathcal{P}$ is said to be weak solution of $(KCS)$ if for all $(\phi, \psi) \in \mathcal{P}$, it satisfies
\begin{equation*}
\begin{split}
m(\|u,v\|^n) &\left(\int_{\Omega} |\nabla u|^{n-2} \nabla u \nabla \phi dx + \int_{\Omega} |\nabla v|^{n-2} \nabla v \nabla \psi dx \right)\\
&= \int_{\Omega} \left(\int_{\Omega} \frac{F(x,u,v)}{|x-y|^{\mu}} dy \right) ( f_1(x,u,v) \phi + f_2(x,u,v) \psi )dx.
\end{split}
\end{equation*}
\end{define}
\noindent We define the energy functional $J$ on $\mathcal{P}$ as	
\begin{equation}\label{EneFunc}
\begin{split}
J(u,v)= \frac{1}{n} M(\|u,v\|^n) -\frac{1}{2} \int_{\Omega} \left( \int_{\Omega} \frac{F(y,u,v)}{|x-y|^{\mu}} dy\right) F(x,u,v) dx.
\end{split}
\end{equation}
Using assumption $(f1)-(f3)$, we get that for any $\epsilon>0, p\geq 1$ and $1 \leq  k < l+1$ there exist constant $C_1, C_2$ such that for any $(x,t,s) \in \Omega \times \mathbb{R}^2$
\begin{equation}\label{bdd1}
|F(x,t,s)| \leq C_1 (|s|^k + |t|^k) + C_2 (|s|^p + |t|^p) \exp((1+\epsilon)(|s|^{\frac{n}{n-1}}+ |t|^{\frac{n}{n-1}})).
\end{equation}
Then by using Sobolev embedding and Hardy-Littlewood Sobolev inequality, we obtain $F(u,v) \in L^q(\Omega \times \Omega)$ for any $q \geq 1$ and the energy functional $J$ is well defined in $\mathcal{P}.$

\subsection{Mountain Pass geometry and Analysis of Palais-Smale sequence}
In this section we show that the energy functional $J$ satisfies the mountain pass geometry and derive the integral estimates of Choquard term by exploiting the weak convergence of Palais-Smale squence in appropriate spaces.

\begin{Lem}\label{MPG}
Assume m and f satisfies $(m2)$ and $(f1)-(f3)$ respectively then
\begin{enumerate}
\item[(i)] There exists $\rho >0$ such that $J(u,v) \geq \sigma$ when $\|(u,v)\|= \rho$, for some $\sigma >0.$
\item[(ii)] There exists a $(\tilde u,\tilde v) \in \mathcal{P}$ such that $J(\tilde{u},\tilde{v}) < 0$ and $\|(\tilde{u},\tilde{v})\| > \rho.$
\end{enumerate}
\end{Lem}
\begin{proof}
Let $(u,v) \in \mathcal{P}$  such that $\|(u,v)\|= \rho$ (to be determined later). Then from \eqref{bdd1}, Proposition \ref{HLS}, Sobolev embedding, H\"older inequality, we have  for any $\epsilon >0$, $p \geq 1$ and $1 \leq k <l+1$ we have
\begin{equation*}
\begin{split}
&\int_{\Omega} \left( \int_{\Omega} \frac{F(y,u,v)}{|x-y|^{\mu}} dy\right) F(x,u,v) dx  \leq C(n, \mu) \|F(x,u,v)\|^2_{L^{\frac{2n}{2n-\mu}}(\Omega)}\\
& \leq \bigg[ C_1 \left( \int_{\Omega} |u|^k + |v|^k \right)^{\frac{2n}{2n-\mu}}\\
& \ \ \ \ \ \ \ \ + C_2 \left( \int_{\Omega} (|u|^p + |v|^p)^{\frac{2n}{2n-\mu}}  \exp\left(\frac{(1+\epsilon) 2n}{2n-\mu}(|u|^{\frac{n}{n-1}} + |v|^{\frac{n}{n-1}} )\right)\right)\bigg]^{\frac{2n-\mu}{n}}\\
& \leq \bigg[ C_1 \left( \|(u,v)\|\right)^{\frac{2nk}{2n-\mu}}\\
& \ \ \ \ \ \ \ \ + C_2 \left( \|(u,v)\| \right)^{\frac{2np}{2n-\mu}} \left(\int_{\Omega} \exp\left(\frac{(1+\epsilon)4 n \|(u,v)\|^{\frac{n}{n-1}}}{2n-\mu}\left(\frac{|u|^{\frac{n}{n-1}} + |v|^{\frac{n}{n-1}}}{\|(u,v)\|^{\frac{n}{n-1}}} \right)\right)\right)^{\frac{1}{2}} \bigg]^{\frac{2n-\mu}{n}}.
\end{split}
\end{equation*}
If we choose $\epsilon>0$ and $\rho$ such that $\frac{4n(1+\epsilon) \rho^{\frac{n}{n-1}}}{2n-\mu} \leq \frac{\alpha_n}{2_n}$, then by using Theorem \ref{MTST0} in above we obtain,
\begin{equation}\label{est11}
\begin{split}
\int_{\Omega} \left( \int_{\Omega} \frac{F(y,u,v)}{|x-y|^{\mu}} dy\right) F(x,u,v) dx  \leq C_3 \|(u,v)\|^{2k} + C_4 \|(u,v)\|^{2p}.
\end{split}
\end{equation}
Niw by using \eqref{est11} and $(m2)$ (for non-degenerate Kirchhoff term), we get
\begin{equation*}
J(u,v) \geq c_0 \frac{\|(u,v)\|^n}{n} - C_3 \|(u,v)\|^{2k} - C_4 \|(u,v)\|^{2p}.
\end{equation*}
So choosing  $k >n/2$, $p>n/2$ and
$\rho$ small enough such that  $J(u,v) \geq \sigma$ when $\|(u,v)\| =\rho$ for some $\sigma>0$ depending on $\rho.$ Similarly for degenerate Kirchhoff term we get,
\begin{equation*}
J(u,v) \geq \frac{\|(u,v)\|^{n(z+1)}}{n} - C_3 \|(u,v)\|^{2k} - C_4 \|(u,v)\|^{2p}
\end{equation*}
and we can choose $2k >n(z+1)$, $2p>n (z+1)$ and
$\tilde{\rho}$ small enough such that $\|(u,v)\| =\tilde{\rho}$ and $J(u,v) \geq \tilde{\sigma}$ for some $\tilde{\sigma}$ depending upon $\tilde{\rho}.$\\
Furthermore, again by using $(m2)$, there exist constant $c_i$, $i=1,2,3$ such that
\begin{equation}\label{est12}
M(\|(u,v)\|^n)\leq\left\{
\begin{split}
& \frac{c_1}{(r+1)} \|(u,v)\|^{n(r+1)} + c_2 \|(u,v)\|^n + c_3 ,\;  \ \ r \neq 1,\\
&  c_1 \ln(\|(u,v)\|^{n}) + c_2 \|(u,v)\|^n + c_3 \; \ \ \ \ \ \ \ \ \ \  r=1,\\
\end{split}
\right.
\end{equation}
for $\|(u,v)\| \geq \tilde{t}$ where
\begin{equation*}
c_3= \left\{
\begin{split}
& M(\tilde{t}) - c_2 \tilde{t} - \frac{c_1}{(r+1)} \tilde{t}^{r+1},\;  \ \ r \neq 1,\\
&  M(\tilde{t})- c_2 \tilde{t} - c_1 \ln(\tilde{t}) \; \ \ \ \ \ \ \ \ \ r=1.\\
\end{split}
\right.
\end{equation*}
Let $(u_0,v_0) \in \mathcal{P}$ such that $u_0 \geq 0, v_0 \geq 0$ and $\|(u_0,v_0)\|=1$. Then by using $(f3)$, there exists $p_i \geq 0$, $i=1,2,3$ and $ K> \frac{n(r+1)}{2}$ such that $F(x,t,s) \geq p_1 |t|^{K} + p_2|s|^{K} - p_3$ and
\begin{equation}\label{est13}
\int_{\Omega} \left( \int_{\Omega} \frac{F(y,\xi u_0,\xi v_0)}{|x-y|^{\mu}} dy\right) F(x,\xi u_0, \xi v_0) dx \geq C_5 \xi^{2K} - C_6 \xi^{K} + C_7.
\end{equation}
Finally by combining \eqref{est12} and \eqref{est13}, we obtain $J(\xi u_0, \xi v_0) \to -\infty$ as $\xi \to \infty.$ Hence there exists $(\tilde{u},\tilde{v}) \in \mathcal{P}$ such that $J(\tilde{u},\tilde{v}) <0$ and $\|(\tilde{u},\tilde{v})\| > \rho.$
\end{proof}

\begin{Lem}\label{PS-bdd}
Every Palais-Smale sequence is bounded in $\mathcal{P}.$
\end{Lem}
\begin{proof}
Let $(u_k, v_k)$ be a Palais-Smale sequence such that $J(u_k,v_k) \to c$ and $J'(u_k,v_k) \to 0$ as $k \to \infty$ for some $c \in \mathbb{R}$. Therefore we have:
\begin{equation}\label{est14}
\begin{split}
\bigg|\frac{M(\|(u_k,v_k)\|^n)}{n}-\frac{1}{2} \int_{\Omega} \left( \int_{\Omega} \frac{F(y,u_k,v_k)}{|x-y|^{\mu}} dy\right) F(x,u_k,v_k) dx\bigg| \to c
\end{split}
\end{equation}
and
\begin{equation}\label{est15}
\begin{split}
&\bigg|m(\|(u_k,v_k)\|^n) \left( \int_{\Omega} |\nabla u_k|^{n-2} \nabla u_k \nabla \phi dx + \int_{\Omega} |\nabla v_k|^{n-2} \nabla v_k \nabla \psi dx \right)\\
 & - \int_{\Omega} \left( \int_{\Omega} \frac{F(y,u_k,v_k)}{|x-y|^{\mu}} dy\right) (f_1(x,u_k,v_k) \phi + f_2(x,u_k,v_k) \psi) dx \bigg| \leq \epsilon_k \|(\phi, \psi)\|.
\end{split}
\end{equation}
Now by using $(f3)$ and $(m3)$, there exists $\eta >\frac{n}{2}, \theta \geq 2n$ such that $$\eta F(x,t,s) \leq t f_{1}(x,t,s) \ \text{and} \ \eta F(x,t,s) \leq s f_{2}(x,t,s) \ \text{for all} \ (x,t,s) \in \Omega \times \mathbb{R}^2$$ and
$$
\frac{1}{n} M(t)- \frac{1}{\theta} m(t) t \ \text{in nonnegative and nondecreasing for} \ t \geq 0.$$
Then by taking $\phi=u_k$ and $\psi= v_k$ in \eqref{est15} along with $(m2)$ (for both degenerate and non-degenerate Kirchhoff terms) and above inequalities, we obtain
\begin{equation}\label{est16}
\begin{split}
J(u_k,v_k) &- \frac{\langle J'(u_k, v_k), (u_k, v_k)\rangle }{ 4 \eta} = \frac{M(\|(u_k,v_k)\|^n)}{n} - \frac{m(\|(u_k,v_k)\|^n)}{4 \eta} \|(u_k, v_k)\|^n\\
&+ \frac{1}{4 \eta} \int_{\Omega} \left( \int_{\Omega} \frac{F(y,u_k,v_k)}{|x-y|^{\mu}} dy\right) (f_1(x,u_k,v_k)  u_k + f_2(x,u_k,v_k) v_k - 2 \eta F(x,u_k,v_k)) dx \\
& \geq \frac{M(\|(u_k,v_k)\|^n)}{n} - \frac{m(\|(u_k,v_k)\|^n)}{4 \eta} \|(u_k, v_k)\|^n \\
&\geq \left(\frac{1}{2n} - \frac{1}{4 \eta}\right) m(\|(u_k, v_k)\|^n) \|(u_k, v_k)\|^n \\
&\geq \left\{
\begin{split}
& c_0 \left(\frac{1}{2n} - \frac{1}{4 \eta}\right) \|(u_k, v_k)\|^n \\
& \ \quad \quad \text{or} \\
&  \left(\frac{1}{2n} - \frac{1}{4 \eta}\right) \|(u_k, v_k)\|^{n+z}.
\end{split}
\right.
\end{split}
\end{equation}
Also, from \eqref{est14} and \eqref{est15}, we get for some constant $C>0$
\begin{equation}\label{est17}
J(u_k,v_k)- \frac{\langle J'(u_k, v_k), (u_k, v_k)\rangle }{ 4 \eta} \leq C \left( 1 + \epsilon_k \frac{\|(u_k,v_k)\|}{ 4 \eta}\right).
\end{equation}
Therefore, by combining \eqref{est16} and \eqref{est17}, we obtain $\{(u_k, v_k)\}$ is bounded in $\mathcal{P}.$
\end{proof}

\begin{Lem}\label{weakconv}
Let $\{(u_k,v_k)\}$ be a Palais-Smale sequence then up to a subsequence
\begin{equation*}
      \left.
\begin{alignedat}{2}
 & \quad \quad \quad \quad \quad \quad |\nabla u_k|^{n-2} \nabla u_k  \rightharpoonup |\nabla u|^{n-2} \nabla u \\
& \quad \quad \quad \quad \quad \quad |\nabla v_k|^{n-2} \nabla v_k  \rightharpoonup |\nabla v|^{n-2} \nabla v
 \end{alignedat}
     \right\} \text{weakly in} \ W_0^{1,n}(\Omega).
\end{equation*}
\end{Lem}
\begin{proof}
From Lemma \ref{PS-bdd}, we know that every Palais-Smale sequence satisfies \eqref{est14} and \eqref{est15} and is bounded in $\mathcal{P}.$ So there exist $u,v \in W_0^{1,n}(\Omega)$ such that up to a subsequence
\begin{equation*}
      \left\{
\begin{alignedat}{2}
 {} & u_k  \rightharpoonup u, v_k  \rightharpoonup v \ {} &&\text{weakly in} \ W_0^{1,n}(\Omega). \\
& u_k  \to u, v_k \to v && \text{strongly in} \ L^q(\Omega) \ \forall q \geq 1 \  \text{and a.e. in} \ \Omega.
 \end{alignedat}
     \right.
\end{equation*}
Since $|u_k|^n+ |\nabla u_k |^n$ and $|v_k|^n+|\nabla v_k|^n$ is bounded in $L^1(\Omega)$, so there exist two radon measures $\mu_1, \mu_2$ and two functions $u_1, v_1 \in (L^{\frac{n}{n-1}}(\Omega))^n$ such that upto a subsequence
$$|u_k|^n+ |\nabla u_k|^n \to \mu_1 \ \text{and} \ |v_k|^n + |\nabla v_k|^n \to \mu_2  \ \text{in the sense of measure and}$$ $$|\nabla u_k|^{n-2} \nabla u_k \rightharpoonup u_1, |\nabla v_k|^{n-2} \nabla v_k \rightharpoonup v_1 \ \text{weakly in} \ (L^{\frac{n}{n-1}}(\Omega))^n \ \text{as} \ k \to \infty.$$
We set $\sigma_1, \sigma_2 >0$ such that $\frac{2n}{2n-\mu}(\sigma_1 + \sigma_2)^{\frac{1}{n-1}} < \frac{\alpha_n}{2}$ and $X_{\sigma_i}= \{ x \in \overline{\Omega}: \mu_i (B_r(x) \cap \overline{\Omega})) \geq \sigma_i), \ \text{for all} \  r>0\}$ for $i=1,2.$ Then $X_{\sigma_i}$ must be finite sets and we claim that for any open and relatively compact subset $K$ of $\overline{\Omega} \setminus (X_{\sigma_1} \cup X_{\sigma_2})$
\begin{equation}\label{Claim1}
\lim_{k \to \infty} \int_{K} \left( \int_{\Omega} \frac{F(y,u_k,v_k)}{|x-y|^\mu} dy\right) f_1(x,u_k,v_k) u_k \to \lim_{k \to \infty} \int_{K} \left( \int_{\Omega} \frac{F(y,u,v)}{|x-y|^\mu} dy\right) f_1(x,u,v) u
\end{equation}
and
\begin{equation}\label{Claim2}
\lim_{k \to \infty} \int_{K} \left( \int_{\Omega} \frac{F(y,u_k,v_k)}{|x-y|^\mu} dy\right) f_2(x,u_k,v_k) v_k \to \lim_{k \to \infty} \int_{K} \left( \int_{\Omega} \frac{F(y,u,v)}{|x-y|^\mu} dy\right) f_2(x,u,v) v.
\end{equation}
Let $x_0 \in K$ and $r_i >0$ be such that $\mu_i(B_{r_i}(x_0) \cap \overline{\Omega}) < \sigma_i$ and consider $\psi_i \in C^{\infty}(\Omega)$ satisfying $0 \leq \psi_i \leq 1$ for $x \in \Omega$, $\psi_i=1$ in $B_{\frac{r_i}{2}}(x_0) \cap \overline{\Omega}$ and $\psi_i =0$ in $\overline{\Omega} \setminus B_{r_i}(x_0)$ for $i=1,2$. Then
\begin{equation*}
\lim_{k \to \infty} \int_{B_{\frac{r_1}{2}}(x_0) \cap {\Omega}} |u_k|^n + |\nabla u_k|^n dx \leq \lim_{k\to \infty}\int_{B_{r_1}(x_0) \cap {\Omega}} (|u_k|^n + |\nabla u_k|^n)\psi_1 dx =\mu_1(B_{r_1}(x_0) \cap \overline{\Omega}) < \sigma_1
\end{equation*}
and
\begin{equation*}
\lim_{k \to \infty} \int_{B_{\frac{r_2}{2}}(x_0) \cap {\Omega}} |v_k|^n + |\nabla v_k|^n dx \leq \lim_{k \to \infty}\int_{B_{r_2}(x_0) \cap  {\Omega}} (|v_k|^n + |\nabla v_k|^n) \psi_2 dx = \mu_2(B_{r_2}(x_0) \cap \overline{\Omega}) < \sigma_2.
\end{equation*}
Then by choosing $k \in \mathbb{N}$ large enough and $r_0:= \min\{r_1,r_2\}$ we get
\begin{equation}\label{est18}
\|(u_k,v_k)\|^n_{\mathcal{Z}(B_{\frac{r_0}{2}}(x_0) \cap {\Omega})}:=\int_{B_{\frac{r_0}{2}}(x_0) \cap {\Omega}}\left( |u_k|^n + |\nabla u_k|^n +|v_k|^n + |\nabla v_k|^n\right) < (\sigma_1 + \sigma_2).
\end{equation}
Now by using \eqref{est18}, Theorem \ref{MTST11} with $\lambda=0$ and choosing $\epsilon>0$ small enough and $q>1$ such that $\frac{2nq}{2n-\mu} (1+\epsilon)(\sigma_1 + \sigma_2)^{\frac{1}{n-1}} \leq \frac{\alpha_n}{2}$ we get the following estimates for $i=1,2$
\begin{equation}\label{est19}
\begin{split}
& \int_{B_{\frac{r_0}{2}}(x_0) \cap {\Omega}}|f_i(x,u_k, v_k)|^{\frac{2nq}{2n-\mu}} dx\\
&=  \int_{B_{\frac{r_0}{2}}(x_0) \cap {\Omega}}|h_i(x,u_k, v_k)|^{\frac{2nq}{2n-\mu}}\exp\left(\frac{2nq}{2n-\mu}(|u_k|^\frac{n}{n-1} +|v_k|^\frac{n}{n-1})\right)dx \\
&  \leq C_\epsilon \int_{B_{\frac{r_0}{2}}(x_0) \cap {\Omega}}\exp\left( \frac{2nq (1+\epsilon)}{2n-\mu}(|u_k|^{\frac{n}{n-1}} +|v_k|^\frac{n}{n-1})\right)dx\\
& \leq C_\epsilon \int_{B_{\frac{r_0}{2}}(x_0) \cap {\Omega}}\exp\left(\frac{2nq}{2n-\mu}(1+\epsilon)(\sigma_1 +\sigma_2)^{\frac{1}{n-1}}\left(\frac{|u_k|^\frac{n}{n-1} + |v_k|^\frac{n}{n-1}}{\|(u_k,v_k)\|_{\mathcal{Z}(B_{\frac{r_0}{2}}(x_0) \cap {\Omega})}^{\frac{n}{n-1}}} \right)\right)dx\leq \tilde{C}_\epsilon
\end{split}
\end{equation}
for some constant $\tilde{C}_\epsilon>0.$ First we prove \eqref{Claim1}, a similar proof provides \eqref{Claim2}. Consider
\begin{equation*}
\begin{split}
 &\int_{B_{\frac{r_0}{2}}(x_0) \cap {\Omega}} \left|\left( \int_\Omega \frac{F(y,u_k,v_k)}{|x-y|^{\mu}}dy\right) f_1(x,u_k, v_k)u_k- \left( \int_\Omega \frac{F(y,u,v)}{|x-y|^{\mu}}dy\right) f_1(x,u,v)u \right|dx\\
 & \leq \int_{B_{\frac{r_0}{2}}(x_0) \cap {\Omega}} \left|\left( \int_\Omega \frac{F(y,u,v)}{|x-y|^{\mu}}dy\right) (f_1(x,u_k,v_k)u_k-f_1(x,u,v)u)\right|dx\\
 & \quad + \int_{B_{\frac{r_0}{2}}(x_0) \cap {\Omega}} \left|\left( \int_\Omega \frac{F(y,u_k,v_k)-F(y,u,v)}{|x-y|^{\mu}}dy\right) f_1(x,u_k,v_k)u_k\right|dx\\
 &:= I_1 + I_2 \;\text{(say)}.
\end{split}
\end{equation*}
From \eqref{bdd1}, \eqref{est19}, H\"older's inequality and asymptotic growth of $f_i$ we obtain that families $\{f_1(x,u_k,v_k) u_k\}$ and $\{f_2(x,u_k,v_k) v_k\}$ are equi-integrable over $B_{\frac{r_0}{2}}(x_0) \cap \overline{\Omega}$ and $\mu \in (0,n)$ gives
\begin{equation}\label{est20}
\int_\Omega \frac{F(y,u,v)}{|x-y|^{\mu}}dy \in L^\infty(\Omega).
\end{equation}
Then \eqref{est20} and Vitali's convergence theorem combined with pointwise convergence of \\ $f_1(x,u_k,v_k) u_k \to f_1(x,u,v) u$ implies $I_1 \to 0.$    Now we show that $I_2 \to 0$ as $k \to \infty.$
Then by using semigroup property of the Riesz potential (see \cite {MS}) and \eqref{est19} we get that for some constant $C>0$ independent of $k$
\begin{align*}
\begin{split}
&\int_{\Omega} \left(\int_{\Omega}\frac{F(y,u_k,v_k)-F(y,u,v)}{|x-y|^{\mu}} dy \right) \chi_{B_{\frac{r_0}{2}} \cap \overline{\Omega}}(x) f_1(x,u_k,v_k) u_k dx \\
&\leq \left(\int_{\Omega} \left( \int_{\Omega}\frac{|F(y,u_k,v_k)- F(y,u,v)| dy }{|x-y|^{\mu}}\right) |F(x,u_k,v_k)- F(x,u,v)| dx \right)^{\frac{1}{2}}\\
&\quad \times \left(\int_{\Omega} \left(\int_{\Omega} \chi_{B_{\frac{r_0}{2}} \cap \overline{\Omega}}(y) \frac{f_1(y,u_k,v_k) u_k}{|x-y|^{\mu}} dy \right) \chi_{B_{\frac{r_0}{2}} \cap \overline{\Omega}}(x) f_1(x,u_k,v_k) u_k dx \right)^{\frac{1}{2}}\\
& \leq C \left(\int_{\Omega} \left( \int_{\Omega}\frac{|F(y,u_k,v_k)- F(y,u,v)| dy }{|x-y|^{\mu}}\right) |F(x,u_k,v_k)- F(x,u,v)| dx \right)^{\frac{1}{2}}.
\end{split}
 \end{align*}
Now we claim that
\begin{align}\label{est21}
\lim_{k \to \infty} \int_{\Omega} \left(\int_{\Omega}\frac{|F(y,u_k,v_k)-F(y,u,v)|}{|x-y|^{\mu}} dy \right) |F(x,u_k,v_k)-F(x,u,v)|  dx =0.
\end{align}
From \eqref{est14} and , \eqref{est15} we get that there exists a constant $C_1, C_2>0$ (independent of $k$) such that
\begin{equation}\label{est22}
\begin{split}
&\int_\Omega \left(\int_\Omega\frac{F(y,u_k,v_k)}{|x-y|^\mu}dy\right)F(x,u_k,v_k)dx \leq C_1,\\
\int_\Omega &\left(\int_\Omega\frac{F(y,u_k,v_k)}{|x-y|^\mu}dy\right)(f_1(x,u_k,v_k)u_k + f_2(x,u_k,v_k) v_k)dx  \leq C_2.
\end{split}
\end{equation}
We argue as along equation $(3.19)$ in Lemma 3.4 in \cite{AGMS}.
Consider
\begin{equation*}
\begin{split}
& \int_{\Omega} \left(\int_{\Omega}\frac{|F(y,u_k,v_k)-F(y,u,v)|}{|x-y|^{\mu}} dy \right) |F(x,u_k,v_k)-F(x,u,v)|  dx\leq\\
& \int_\Omega\left(\int_\Omega\frac{|F(y,u_k,v_k)\chi_{A}(y)-F(y,u,v)\chi_{B}(y)}{|x-y|^\mu}dy\right)|F(x,u_k,v_k)\chi_{A}(x)-F(x,u,v)\chi_{B}(x)|dx\\
& + 2 \int_{\Omega} \left(\int_{\Omega}\frac{(F(y,u_k,v_k) \chi_A(y)+ F(y,u,v)\chi_B(y)+F(y,u,v) \chi_D(y))}{|x-y|^{\mu}} dy \right) F(x,u_k,v_k) \chi_C(x) dx\\
&  + 2 \int_{\Omega} \left(\int_{\Omega}\frac{(F(y,u_k,v_k)\chi_A(y)+ F(y,u,v) \chi_B(y))}{|x-y|^{\mu}} dy \right) F(x,u,v) \chi_D(x) dx\\
&+ \int_{\Omega} \left(\int_{\Omega}\frac{F(y,u_k,v_k) \chi_C(y)}{|x-y|^{\mu}} dy \right) F(x,u_k,v_k) \chi_C(x) dx\\
& + \int_{\Omega} \left(\int_{\Omega}\frac{F(y,u,v) \chi_D(y)}{|x-y|^{\mu}} dy \right) F(x,u,v) \chi_D(x) dx:= I_3 +I_4 +I_5+ I_6 +I_7.
\end{split}
\end{equation*}
where for a fixed $M>0$ $$A= \{x \in \Omega: |u_k| \leq M \ \text{and} \ |v_k| \leq M \}, \ \ B= \{x \in \Omega: |u| \leq M \ \text{and} \ |v| \leq M \},$$
$$C= \{x \in \Omega: |u_k| \geq M \ \text{or} \ |v_k| \geq M \} \ \ \text{and} \ D= \{x \in \Omega: |u| \geq M \ \text{or} \ |v| \geq M \}.$$
Now using \eqref{est22}, $(f4)$, semigroup property of the Riesz Potential we obtain $I_j= o(M)$ for $j = 4,\ldots,7$, when $M$ is large enough and from Lebesgue dominated convergence theorem we obtain $I_3 \to 0$ as $k \to \infty.$ Hence \eqref{est21} holds and $I_2 \to 0$ as $k \to \infty.$ Now to conclude \eqref{Claim1} and \eqref{Claim2}, we repeat this procedure over a finite covering of balls using the fact that $K$ is compact.
Now the remaining proof can be done by using the same arguments as in Lemma $3.4$ in \cite{AGMS}.
\end{proof}
\begin{Lem}\label{Choqest}
Let $\{(u_k,v_k)\}$ be a Palais-Smale sequence for the energy functional $J$. Then there exists $(u,v) \in \mathcal{P}$ such that upto a subsequence
\begin{equation*}
\int_{\Omega} \left(\int_{\Omega} \frac{F(x,u_k,v_k)}{|x-y|^\mu} dy \right) f_i(x,u_k,v_k) \phi dx \to \int_{\Omega} \left(\int_{\Omega} \frac{F(x,u,v)}{|x-y|^\mu} dy \right) f_i(x,u,v) \phi dx
\end{equation*}
for all $\phi \in C_c^{\infty}(\Omega)$ and $i=1,2$
and
\begin{equation*}
\left(\int_{\Omega} \frac{F(x,u_k,v_k)}{|x-y|^\mu} dy \right) F(x,u_k,v_k) \to  \left(\int_{\Omega} \frac{F(x,u,v)}{|x-y|^\mu} dy \right) F(x,u,v)  dx \ \ \text{in} \ L^1(\Omega).
\end{equation*}
\end{Lem}
\vspace{.2cm}
\noindent The proof of the above Lemma follows from similar arguments as in Lemma $3.5$ and Lemma $3.6$ in \cite{AGMS}.
\vspace{.4cm}\\
\noindent Now we define the Mountain pass critical level and associated Nehari Manifold as
\begin{equation*}
l^*= \inf_{\gamma \in \Gamma} \max_{t \in [0,1]} J(\gamma(t))\ \text{where} \ \Gamma= \{ \gamma \in C([0,1], \mathcal{P}): \gamma(0) =0, J(\gamma(1)) <0 \}
\end{equation*} and
$$N =\{ (u,v) \in (W_0^{1,n}(\Omega) \setminus \{0\} )^2 : \langle J'(u,v), (u,v) \rangle  =0\}.$$
\begin{Lem}\label{cricest}
Let $ l^{**} = \inf_{u \in N} J(u)$. Assume $(m3)$, $(f3)$ and for some $q>2$
\begin{equation}\label{mainassu}
\lim_{t,s \to \infty} \frac{(f_1(x,t,s)t + f_2(x,t,s)s) F(x,t,s)}{\exp( q (|t|^{\frac{n}{n-1}} + |s|^{\frac{n}{n-1}}))} = \infty \ \text{uniformly in} \ x \in \overline{\Omega}
\end{equation} holds then $$l^{*} \leq l^{**} \ \text{and} \ 0< l^* < \frac{1}{n} M\left(\left(\left(\frac{2n-\mu}{2n}\right)\frac{\alpha_n}{2_n} \right)^{n-1} \right).$$
\end{Lem}
\begin{proof}
Let $(u,v) \in N$ and $h:(0,\infty) \to \mathbb{R}$ such that $h(t)=J((tu,tv)).$ Then
$$h^\prime(t)= {m(\|(tu,tv)\|^n)} t^{n-1} \|(u,v)\|^n - \int_\Omega \left(\int_\Omega \frac{F(y,tu,tv)}{|x-y|^{\mu}}dy\right)(f_1(x,tu,tv)u + f_2(x,tu,tv)v)dx.$$
Since $(u,v) \in N$, we get
\begin{align*}
h^\prime(t) &= h^\prime(t) - t^{2n-1} \langle J'(u,v), (u,v)\rangle =  t^{2n-1} \left( \frac{m(\|(tu,tv)\|^n)}{t^n\|(u,v)\|^n}- \frac{m(\|(u,v)\|^n)}{\|(u,v)\|^n}\right) \|(u,v)\|^{2n}\\
&\quad + t^{2n-1}\bigg[\int_\Omega\left(  \int_\Omega \frac{F(y,u,v)}{|x-y|^{\mu}}dy \right){ (f_1(x,u,v) u + f_2(x,u,v) v)}dx \\
& \quad \quad \quad \quad \quad \quad \quad \quad \quad - \int_\Omega\left(\int_{\Omega}\frac{F(y,tu,tv)}{ t^{2n}|x-y|^{\mu}}dy\right) (f_1(x,tu,tv) tu +f_2(x,tu,tv) tv)~dx \bigg].
\end{align*}
Now $(f3)$ implies,
for any $(x,s) \in \Omega \times \mathbb{R}^+$, the map $r \mapsto r f_1(x,r,s) - n F(x,r,s)$  and for any $(x,r) \in \Omega \times \mathbb{R}^+$, the map $s \mapsto s f_2(x,r,s) - n F(x,r,s)$ is increasing on $\mathbb{R}^+$. Using this we get $r f_1(x,r,s) - n F(x,r,s) \geq 0$ and $s f_2(x,r,s) - n F(x,r,s) \geq 0$ for all $(x,r,s) \in \Omega \times \mathbb{R}^2$ which implies
$$ t \mapsto \frac{F(x,tu,tv)}{t^n} \ \text{is non-decreasing for } t>0.$$
Then for $0<t\leq1$, $x \in \Omega$ and by using $(m3)$ and $(f3)$, we obtain
\begin{align*}
h^\prime(t) &\geq  t^{2n-1} \left( \frac{m(\|(tu,tv)\|^n)}{t^n\|(u,v)\|^n}- \frac{m(\|(u,v)\|^n)}{\|(u,v)\|^n}\right) \|(u,v)\|^{2n}\\
& + t^{2n-1}\bigg[\int_\Omega \left( \int_{\Omega} \frac{F(y,u,v)}{|x-y|^\mu} dy\right) \left(\left( \frac{f_1(x,u,v)u}{u^{n}} -\frac{f_1(x,tu,tv)tu}{(tu)^{n}} \right) u^n(x)\right.\\
& \quad  \quad \quad \quad \quad \quad \quad \quad \quad \quad \quad \quad \quad + \left.\left( \frac{f_2(x,u,v)v}{v^{n}} -\frac{f_2(x,tu,tv)tv}{(tv)^{n}} \right) v^n(x)\right)
\bigg] \geq 0.
\end{align*}
This gives that $h^\prime(t)\geq 0$ for $0<t\leq1$ and $h^\prime(t)<0$ for $t>1$. Hence $J(u,v)= \max_{t\geq 0} J(tu,tv)$. Now we define $g:[0,1] \to \mathcal{P}$ as $g(t)=(t_0u, t_0 v)t$ where $t_0>1$ is such that $J((t_0 u, t_0v))<0$. So, $g \in \Gamma$ which gives
\[l^* \leq \max_{t\in[0,1]}J(g(t)) \leq \max_{t\geq 0} J(tu,tv)=J(u,v). \]
Since $u \in N$ is arbitrary, we get $l^* \leq l^{**}$.
For $u,v\not\equiv 0$, $J(tu,tv) \to -\infty$ as $t\to \infty$ (from Lemma~\ref{MPG}) and by definition $l^* \leq \max_{t\in[0,1]} J(tu,tv)$ for $(u,v) \in (W^{1,n}_0(\Omega)\backslash\{0\})^2$ satisfying $J(u,v)<0$. So, it is enough to prove that there exists a $(w_1,w_2) \in \mathcal{P}$ such that $\|(w_1,w_2)\|=1$ and
\begin{equation}\label{est23}
\max_{t\in[0,\infty)} J(tw_1,tw_2) < \frac{1}{n} M\left(\left(\left(\frac{2n-\mu}{2n}\right)\frac{\alpha_n}{2_n} \right)^{n-1} \right).
\end{equation}
To prove this, we consider the sequence of functions $\{(U_k, V_k)\}$ as defined in the proof of Theorem \ref{MTST00}
such that supp$(U_k)$, supp$(V_k) \subset B_\rho(0)$ and $\|(U_k, V_k)\| =1$ for all $k$. So we claim that there exists a $k \in \mathbb{N}$ such that \eqref{est23} is satisfied for $w_1= U_k$ and $w_2=V_k.$\\
We proceed by contradiction, suppose this is not true then for all $k \in \mathbb{N}$ there exists a $t_k>0$ such that \eqref{est23} does not holds {\it i.e.} $$\max_{t \in [0, \infty)} J(tU_k,tV_k) = J(t_k U_k, t_k V_k) \geq  \frac{1}{n} M\left(\left(\left(\frac{2n-\mu}{2n}\right)\frac{\alpha_n}{2_n} \right)^{n-1} \right).$$ Since $J((t U_k, t V_k)\to -\infty$ as $t\to \infty$ uniformly in $k$ therefore $\{t_k\}$ must be a bounded sequence in $\mathbb{R}$. Then from \eqref{EneFunc}, $\|U_k, V_k\|=1$ and monotonicity of $M$, we obtain
\begin{equation}\label{est24}
\left(\frac{2n-\mu}{2n}\right)\frac{\alpha_n}{2_n} \leq t_k^\frac{n}{n-1}.
\end{equation}
Since $\frac{d}{dt}(J((tU_k, tV_k))|_{t=t_k}=0$ and $ \int_{B_{\rho/k}}\int_{B_{\rho/k}} \frac{dxdy}{|x-y|^\mu} \geq C_{\mu, n} \left(\frac{\rho}{k}\right)^{2n-\mu}$ then by using
\eqref{mainassu}, for $k\in \mathbb N$ large enough we obtain
\begin{equation*}
\begin{split}
m(t_k^n)t_k^n &= \int_\Omega \left(\int_\Omega \frac{F(y,t_kU_k, t_k V_k)}{|x-y|^{\mu}}dy\right)(f_1(x,t_k U_k, t_k V_k)t_k U_k + f_2(x,t_k U_k, t_k V_k)t_k V_k) dx\\
& \geq \int_{B_{\rho/k}}\left(\int_{B_{\rho/k}}\frac{F(y,t_kU_k, t_k V_k)}{|x-y|^\mu}dy\right)(f_1(x,t_k U_k, t_k V_k)t_k U_k + f_2(x,t_k U_k, t_k V_k)t_k V_k) dx.\\
& \geq \exp\left( q (c_1^{\frac{n}{n-1}} + c_2^{\frac{n}{n-1}}) \left( \frac{t_k^{\frac{n}{n-1}}(\log k)}{\omega_{n-1}^{\frac{1}{n-1}}}\right)\right)\int_{B_{\rho/k}}\int_{B_{\rho/k}} \frac{dxdy}{|x-y|^\mu} \\
& \geq \tilde{C}_{\mu,n}  k^{\left(\frac{q \left(c_1^{\frac{n}{n-1}} + c_2^{\frac{n}{n-1}}\right)t_k^{\frac{n}{n-1} }}{\omega_{n-1}^{\frac{1}{n-1}}} - (2n-\mu)\right)}.
\end{split}
\end{equation*}
Hence by using the fact that $(c_1^{\frac{n}{n-1}} + c_2^{\frac{n}{n-1}}) = 2_n$, $t_k^n$ is bounded, $q>2$ and \eqref{est24}, we arrive at a contradiction by taking $k$ large enough. 
\end{proof}

\noindent \textbf{Proof of Theorem \ref{existence}}:\\
Let $\{(u_k, v_k)\}$ denotes a Palais Smale sequence at the mountain pass critical level $l^*$. 
Then by Lemma \ref{PS-bdd}  there exists a $u_0, v_0 \in W^{1,n}_0(\Omega)$ such that up to a subsequence
$u_k \rightharpoonup u_0$, $v_k \rightharpoonup v_0$ weakly in $W^{1,n}_0(\Omega)$ as $k \to \infty$. We prove our main result in several steps.\vspace{0.2cm}\\
\noindent \textbf{Step 1:} Positivity of $u_0, v_0$.\\
If $u_0 = v_0 \equiv 0$ (or either one of them) then using Lemma \ref{Choqest}, we infer that
\[\int_\Omega \left(\int_\Omega \frac{ F(y,u_k, v_k)}{|x-y|^{\mu}}dy\right)F(x,u_k, v_k) dx  \to 0\; \text{as}\; k \to \infty\]
and which further gives that $ \lim_{k \to \infty} J(u_k, v_k) = \frac{1}{n}\lim_{k\to\infty} M(\|(u_k, v_k)\|^n) = l^*$. Now in the light of Lemma \ref{cricest} and monotonicity of $M$, we obtain
\[\frac{2n}{2n-\mu}\|(u_k, v_k)\|^{\frac{n}{n-1}} <\frac{\alpha_{n}}{ 2_n}  \]
for large enough $k$.
Now, this implies that $\sup_{k} \int_{\Omega} f_i(x,u_k, v_k)^q dx < +\infty$ for some $q >\frac{2n}{2n-\mu}$, $i=1,2$. Along with  \eqref{bdd1}, Theorem \ref{MTST0}, the Hardy-Littlewood-Sobolev inequality and the Vitali's convergence theorem, we also obtain
\[\int_{\Omega}\left( \int_\Omega\frac{F(y,u_k, v_k)}{|x-y|^{\mu}}dy\right)(f_1(x,u_k, v_k)u_k + f_2(x,u_k, v_k)v_k) dx \to 0 \;\text{as}\; k \to \infty.\]
Hence $\lim_{k\to \infty}\langle J^\prime((u_k, v_k)), (u_k, v_k) \rangle=0$ gives $\lim_{k\to \infty}m(\|(u_k, v_k)\|^n)\|(u_k, v_k)\|^n=0$. Now from (m2), we obtain $\lim_{k \to \infty} \|(u_k, v_k)\|^n =0$. Thus using Lemma \ref{Choqest}, it must be that $\lim_{k \to \infty} J(u_k, v_k)=0 =l^*$ which contradicts $l^*>0$. Thus $u_0, v_0 \not \equiv 0$ and there exists a constant $\Upsilon >0$ such that up to a subsequence $\|u_k\|^n + \|v_k\|^n \to \Upsilon^n$ as $k \to \infty$. Then from Lemma~\ref{weakconv} and Lemma~\ref{Choqest}, we get as $k \to \infty$,
\begin{align*}
\int_\Omega \left(\int_\Omega \frac{F(y,u_k, v_k)}{|x-y|^{\mu}}dy\right)&(f_1(x,u_k,v_k) \varphi + f_2(x,u_k,v_k) \psi) dx \to\\
& \int_\Omega \left(\int_\Omega \frac{F(y,u_0, v_0)}{|x-y|^{\mu}}dy\right)(f_1(x,u_0, v_0) \varphi+ f_2(x,u_0,v_0) \psi)dx
\end{align*}
and
\begin{equation}\label{est25}
\begin{split}
&m(\Upsilon^n)  \int_\Omega (|\nabla u_0|^{n-2}\nabla u_0\nabla \varphi + |\nabla v_0|^{n-2}\nabla v_0\nabla \psi dx \\
&= \int_\Omega \left(\int_\Omega \frac{F(y,u_0, v_0)}{|x-y|^{\mu}}dy\right) (f_1(x,u_0, v_0)\varphi  + f_2(x,u_0, v_0) \psi)dx, \; \text{for all}\; \varphi, \psi \in W^{1,n}_0(\Omega).
\end{split}
\end{equation}
In particular, taking $\varphi = u_0^-$ and $\psi =0$ (similarly $\varphi = 0$ and $\psi =v_0^-$) in \eqref{est25} we get $m(\Upsilon^n)\|u_0^-\|=0$ (similarly $m(\Upsilon^n)\|v_0^-\|=0$) and together with assumption (m2) implies $u_0^-=0$ ( $v_0^-=0$) a.e. in $\Omega$. Therefore $u_0, v_0 \geq 0$ a.e. in $\Omega$.\\
From Theorem \ref{MTST0} and  H\"older inequality we get, $$\left(\int_\Omega \frac{F(y,u_0, v_0)}{|x-y|^{\mu}}dy \right)(f_1(x,u_0, v_0) + f_2(x, u_0,v_0)) dx \in L^q(\Omega)$$ for $1 \leq q <\infty$. By elliptic regularity results and strong maximum principle, we finally get that $u_0, v_0>0$ in $\Omega$.\\

\noindent \textbf{Step 2:}\ $ m(\|u_0, v_0\|^n)\|(u_0, v_0)\|^n \geq \displaystyle \int_\Omega \left(\int_\Omega \frac{F(y,u_0,v_0)}{|x-y|^{\mu}}dy \right)(f_1(x,u_0, v_0)u_0 + f_2(x, u_0, v_0) v_0)~dx.$\\
Suppose by contradiction
\[ m(\|u_0, v_0\|^n)\|(u_0, v_0)\|^n <\int_\Omega \left(\int_\Omega \frac{F(y,u_0, v_0)}{|x-y|^{\mu}}dy \right)(f_1(x,u_0, v_0)u_0 + f_2(x, u_0, v_0) v_0)dx\]
which implies that $\langle J^\prime(u_0, v_0),(u_0, v_0) \rangle <0$. For $t>0$ small enough, using $(f3)$ and $(f5)$ we have that
\begin{align*}
&\langle J^\prime(t u_0, t v_0), (u_0, v_0) \rangle  \geq  m_0t^{n-1}\|u_0, v_0\|^n \\
&- \frac{1}{2n}\int_\Omega \left(\int_\Omega \frac{f_1(y,tu_0, t v_0)tu_0 + f_2(x,tu_0, v_0) tv_0}{|x-y|^{\mu}}dy \right)
(f_1(x,tu_0, t v_0)u_0 + f_2(x, u_0,v_0) v_0)~dx \\
& \geq m_0t^{n-1}\|u_0, v_0\|^n
- \frac{t^{2\gamma +1}}{2n}\int_\Omega \left(\int_\Omega \frac{((u_0^\gamma + v_0^\gamma) u_0 + (u_0^\gamma + v_0^\gamma) v_0)}{|x-y|^{\mu}}dy \right)
((u_0^\gamma + v_0^\gamma) u_0 + (u_0^\gamma + v_0^\gamma) v_0)~dx \\
& \geq 0.
\end{align*}
Thus there exists a $t_*\in (0,1)$ such that $\langle J^\prime(t_*u_0, t_* v_0),(u_0 , v_0) \rangle=0$ {\it i.e.} $(t_*u_0, t_* v_0) \in N$. So using Lemma \ref{cricest}, $(m3)$ and $(f3)$ we get
\begin{align*}
&l^* \leq l^{**} \leq J((t_*u_0, t_* v_0)) = J(t_*u_0, t_* v_0) - \frac{1}{2n}\langle J^\prime(t_*u_0, t_* v_0),(u_0, v_0) \rangle\\
& = \frac{M(\|t_*u_0, t_* v_0\|^n)}{n} -\frac12 \int_\Omega \left(\int_\Omega \frac{ F(y,t_*u_0, t_* v_0)}{|x-y|^{\mu}}dy\right)F(x,t_*u_0, t_* v_0)dx\\
& \quad -\frac{1}{2n} m(\|t_*u_0, t_* v_0\|^n)\|(t_*u_0, t_* v_0)\|^n\\
& \quad + \frac{1}{2n}\int_\Omega \left(\int_\Omega \frac{F(y,t_*u_0, t_* v_0)}{|x-y|^{\mu}}dy\right) (f_1(x,t_*u_0, t_* v_0)t_*u_0 + f_2(x, t_* u_0, t_*  v_0)dx\\
& < \frac{M(\|u_0, v_0\|^n)}{n}  -\frac{1}{2n}m(\|(u_0, v_0)\|^n)\|(u_0, v_0)\|^n \\
& + \frac{1}{2n}\int_\Omega \left(\int_\Omega \frac{ F(y,t_*u_0, t_* v_0)}{|x-y|^{\mu}}dy\right) (f_1(x,t_*u_0, t_0 v_0)t_*u_0 + f_2(x, t_* u_0, t_* v_0) -  nF(x,t_*u_0, t_* v_0))dx\\
& \leq  \frac{M(\|u_0, v_0\|^n)}{n}  -\frac{1}{2n}m(\|u_0, v_0\|^n)\|u_0, v_0\|^n \\
& \quad \quad \quad  + \frac{1}{2n}\int_\Omega \left(\int_\Omega \frac{ F(y,u_0, v_0)}{|x-y|^{\mu}}dy\right)(f_1(x,u_0, v_0)u_0 + f_2(x, u_0, v_0) -nF(x,u_0, v_0)) dx\\
& \leq \liminf_{k \to \infty} \left( \frac{M(\|u_k, v_k\|^n)}{n}  -\frac{1}{2n}m(\|(u_k, v_k)\|^n)\|(u_k, v_k)\|^n \right.\\
&\quad \quad \quad  \left.+ \frac{1}{2n}\int_\Omega \left(\int_\Omega \frac{ F(y,u_k, v_k)}{|x-y|^{\mu}}dy\right)(f_1(x,u_k, v_k)u_k + f_2(x, u_k, v_k)-nF(x,u_k, v_k))dx\right)\\
& = \liminf_{k \to \infty} \left( J(u_k, v_k) - \frac{1}{2n}\langle J^\prime(u_k, v_k),(u_k, v_k) \rangle\right) = l^*.
\end{align*}
This gives a contradiction and completes the proof of \textbf{Step 2}. Similar arguments follows for the degenerate case also using $(m3)$.\\

\noindent \textbf{Step 3:} $J(u_0, v_0)= l^*$.\\
Using the weakly lower semicontinuity of norms in $\lim_{k \to \infty}J(u_k, v_k)= l^*$ and Lemma \ref{Choqest} we obtain $J(u_0, v_0) \leq l^*$. If $J(u_0, v_0)< l^*$ then it must be
$\lim_{k \to \infty} M(\|u_k, v_k\|^n) > M(\|u_0, v_0\|^n).$
Then continuity and motonicity of $M$ implies $\Upsilon^n > \|u_0, v_0\|^n$
and
\begin{equation}\label{est26}
M(\Upsilon^n) = n \left( l^* + \frac12 \int_\Omega \left(\int_\Omega \frac{ F(y,u_0, v_0)}{|x-y|^{\mu}}dy\right)F(x,u_0, v_0)dx\right).
\end{equation}
Define the sequence of functions $$(\tilde{u}_k, \tilde{v}_k )= \left(\frac{u_k}{\|u_k, v_k\|}, \frac{v_k}{\|u_k, v_k\|}\right)$$ such that $\|\tilde{u}_k, \tilde{v}_k\|=1$ and $\tilde{u}_k, \tilde{v}_k \rightharpoonup (\tilde{u}_0, \tilde{v}_0) = \left(\frac{u_0}{\Upsilon}, \frac{v_0}{\Upsilon} \right)$  weakly in $\mathcal{P}$ and $\|u_0, v_0\|< \Upsilon$. From Theorem \ref{MTST01}, we have that
\begin{equation}\label{est27}
\sup_{ k \in \mathbb{N}} \int_\Omega \exp \left( p(|\tilde{u}_k|^{\frac{n}{n-1}} + |\tilde{v}_k|^{\frac{n}{n-1}})\right) dx <+\infty,\; \text{for} \; 1<p<\frac{\alpha_n}{2_n(1-\|\tilde{u}_0,\tilde{v}_0\|^n)^{\frac{1}{n-1}}}.
\end{equation}
Then from $(m3)$, Claim (1) and Lemma \ref{cricest} we obtain
\begin{align*}
J(u_0, v_0)&= \frac{M(\|u_0, v_0\|^n)}{n}- \frac{m(\|u_0,v_0\|^n)\|u_0, v_0\|^n}{2n}\\
&+ \frac{1}{2n}\int_\Omega \left(\int_\Omega \frac{ F(y,u_0, v_0)}{|x-y|^{\mu}}dy\right)(f_1(x,u_0, v_0)u_0+ f_2(x, u_0,v_0) v_0-nF(x,u_0, v_0))dx\geq 0.
\end{align*}
and from \eqref{est26} we get
\begin{align*}
M(\Upsilon^n) = nl^* - n J(u_0, v_0) +M(\|u_0, v_0\|^n) < M\left(\left(\left(\frac{2n-\mu}{2n}\right)\frac{\alpha_n}{2_n} \right)^{n-1}\right) + M(\|u_0, v_0\|^n)
\end{align*}
which further implies together with (m1) that
\begin{equation*}
\Upsilon^n < \frac{1}{1-\|\tilde{u}_0, \tilde{v}_0\|^n} \left(\left(\frac{2n-\mu}{2n}\right)\frac{\alpha_n}{2_n} \right)^{n-1}.
\end{equation*}
Thus for $k\in \mathbb{N}$ large enough it is possible $b>1$ but close to $1$ such that
\[\frac{2n}{2n-\mu}                                                 ~b~ \|u_k, v_k\|^{\frac{n}{n-1}} \leq  \frac{\alpha_n}{2_n(1-\|\tilde{u}_0, \tilde{v}_0\|^n)^{\frac{1}{n-1}}}.\]
Therefore from \eqref{est27} we conclude that
\begin{equation}\label{est28}
 \int_\Omega \exp\left(\frac{2n}{2n-\mu}b(|u_k|^{\frac{n}{n-1}}+ |v_k|^{\frac{n}{n-1}}\right) \leq C
 \end{equation}
and
\begin{align*}
\int_\Omega \left(\int_\Omega \frac{ F(y,u_k, v_k)}{|x-y|^{\mu}}dy\right)&(f_1(x,u_k, v_k)u_k + f_2(x, u_k, v_k) v_kdx \to \\
&\int_\Omega \left(\int_\Omega \frac{ F(y,u_0, v_0)}{|x-y|^{\mu}}dy\right)(f_1(x,u_0, v_0)u_0 + f_2(x, u_0, v_0) v_0dx.
\end{align*}
This implies $(u_k, v_k) \to (u_0, v_0)$ strongly in $\mathcal{P}$ and hence $J(u_0, v_0) = l^*$ which is a contradiction. Hence, $J(u_0, v_0) = l^* = \lim_{k \to \infty} J(u_k, v_k)$ and $\|(u_k, v_k)\| \to \Upsilon$ implies $\Upsilon = \|(u_0, v_0)\|$. Then finally we have
\begin{equation*}
\begin{split}
m(\|u_0,v_0\|^n) &\left(\int_{\Omega} |\nabla u_0|^{n-2} \nabla u_0 \nabla \phi dx + \int_{\Omega} |\nabla v_0|^{n-2} \nabla v_0 \nabla \psi dx \right)\\
&= \int_{\Omega} \left(\int_{\Omega} \frac{F(x,u_0,v_0)}{|x-y|^{\mu}} dy \right) ( f_1(x,u_0,v_0) \phi + f_2(x,u_0,v_0) \psi )dx
\end{split}
\end{equation*}
for all $\phi, \psi \in W_0^{1,n}(\Omega)$. This completes the proof.
\qed
\section{Extensions and relative problems}
The results of this paper can be extended in various directions. Let us mention here some obvious generalizations:\vspace{0.2cm} \\
\textbf{1}: The class of system (KCS) can be extended to the following fractional Kirchhoff-Choquard system involving singular weights: 
\begin{equation*}
  (F)\left\{
         \begin{alignedat}{2}
             {} -m\left( \int_{\mathbb{R}^n} \int_{\mathbb{R}^n} \frac{|u(x)-u(y)|^{n/s}}{|x-y|^{2n}} dxdy\right)\Delta_{n/s}^s u
             & {}=  \left(\int_{\Omega} \frac{F(y,u,v)}{|y|^{\alpha}|x-y|^\mu}dy\right)\frac{f_1(x,u,v)}{|x|^{\alpha}}\; 
             && \mbox{ in }\, \Omega ,
             \\
             -m\left( \int_{\mathbb{R}^n} \int_{\mathbb{R}^n} \frac{|v(x)-v(y)|^{n/s}}{|x-y|^{2n}} dxdy\right)\Delta_{n/s}^s v
             & {}=  \left(\int_{\Omega} \frac{F(y,u,v)}{|y|^{\alpha} |x-y|^\mu}dy\right)\frac{f_2(x,u,v)}{|x|^{\alpha}}\;
             && \mbox{ in }\, \Omega ,
             \\
             u,v & {}= 0
             && \mbox{ in }\, \mathbb{R}^n \setminus \Omega,
          \end{alignedat}
     \right.
\end{equation*}
where $(-\Delta)^s_{n/s}$ is the $n/s$ fractional Laplace operator, $s \in (0,1)$, $n \geq 1$, $\mu \in (0,n), 0< \alpha < \min\{\frac{n}{2}, n-\mu\}$, $\Omega \subset \mathbb{R}^n$  is a smooth bounded domain, $m : \mathbb{R}^+ \to \mathbb{R}^+$ and $F: \Omega \times  \mathbb{R}^2 \to \mathbb{R}$ is a continous functions where $F$ behaves like $\exp(|u|^{\frac{n}{n-s}} + |v|^\frac{n}{n-s})$ as $|u|, |v| \to \infty.$ \vspace{0.2cm}\\
We conjecture that the following Moser-Trudinger inequality holds in case fractional Sobolev space (counterpart of Theorem \ref{MTST0}):
Define $\mathcal{L}:= X_0 \times X_0$ endowed with the norm $$\|(u,v)\|_{\mathcal{L}}:= \left(\|u\|_{X_0}^{n/s} + \|u\|_{X_0}^{n/s} \right)^{\frac{s}{n}}$$ where  
$$X_0:= \{u \in W^{s,n/s}(\mathbb{R}^n): u=0 \ \text{in} \ \mathbb{R}^n \setminus \Omega\}$$ endowed with the norm $$\|u\|_{X_0} = \left(\int_{\mathbb{R}^{2n} \setminus (\Omega \times \Omega)^c} \frac{|u(x)-u(y)|^{\frac{n}{s}}}{|x-y|^{2n}} dxdy \right)^{n/s}$$
\begin{thm}\label{MTST011}
For $(u,v) \in \mathcal{L}$, $n/s > 2$ and $\Omega \subset \mathbb{R}^n$ is a smooth bounded domain, we have $$\int_{\Omega} \exp\left({ \Pi \left(|u|^{\frac{n}{n-s}}+|v|^{\frac{n}{n-s}}\right)}\right) dx < \infty$$ for any $ \Pi >0.$ Moreover,
\begin{equation}\label{main1121}
\sup_{\|(u,v)\|_{\mathcal{L}}=1}\int_{\Omega} \exp\left( \Pi \left(|u|^{\frac{n}{n-s}}+|v|^{\frac{n}{n-s}}\right)\right) dx < \infty ,\ \text{provided} \ \   \Pi \leq \frac{\alpha_{n,s}}{2_{n,s}}
\end{equation}
where $\alpha_{n,s}= \frac{n}{\omega_{n-1}} \left(\frac{\Gamma\left(\frac{n-s}{2}\right)}{\Gamma\left(\frac{s}{2} \right)2^s \pi^{n/2}} \right)^{\frac{-n}{n-s}}$, $2_{n,s}= 2^{\frac{n-2s}{n-s}}.$ Furthermore if $\Pi > \frac{\alpha_{n,s}^*}{2_{n,s}}$, then there exists a pair $(u,v) \in \mathcal{L}$ with $\|(u,v)\|_{\mathcal{L}}=1$ such that the supremum in \eqref{main1121} is infinite.
\end{thm}
\noindent Using Theorem \ref{MTST011}, doubly weighted Hardy-Littlewood-Sobolev inequality, we can prove the existence and multiplicity of solutions for the problem $(F)$ (see \cite{AGMS, AGMS1}). \vspace{0.2cm}\\
\textbf{2}: We infer that similiar methods can be used to the following Kirchhoff-Choquard system for the Polyharmonic operator: 
\begin{equation*}
     (P)
     \left\{
         \begin{alignedat}{2}
             {} -M(\int_\Omega|\nabla^m u|^2dx)\Delta^m u
             & {}=  \left(\int_{\Omega} \frac{F(y,u,v)}{|y|^{\alpha}|x-y|^\mu}dy\right)\frac{f_1(x,u,v)}{|x|^{\alpha}},\; u>0
             && \quad\mbox{ in }\, \Omega ,
             \\
             -M(\int_\Omega|\nabla^m v|^2dx)\Delta^m v
             & {}=  \left(\int_{\Omega} \frac{F(y,u,v)}{|y|^{\alpha} |x-y|^\mu}dy\right)\frac{f_2(x,u,v)}{|x|^{\alpha}},\; v>0
             && \quad\mbox{ in }\, \Omega ,
             \\
             u= \nabla u= \dots = \nabla^{m-1} u  & {}= 0
             && \quad\mbox{ on }\, \partial\Omega,
             \\
             v= \nabla v= \dots = \nabla^{m-1} v  & {}= 0
             && \quad\mbox{ on }\, \partial\Omega,
          \end{alignedat}
     \right.
\end{equation*}
where $n=2m$, $\mu \in (0,n), 0< \alpha < \min\{\frac{n}{2}, n-\mu\}$, $\Omega \subset \mathbb{R}^n$  is a smooth bounded domain, $M : \mathbb{R}^+ \to \mathbb{R}^+$ and $F: \Omega \times  \mathbb{R}^2 \to \mathbb{R}$ is a continous functions where $F$ behaves like $\exp(|u|^{\frac{n}{n-m}} + |v|^\frac{n}{n-m})$ as $|u|, |v| \to \infty.$ The vectorial polyharmonic operator $\Delta^{m}_{\frac{n}{m}}$ is defined as 
\begin{equation*}
\Delta^m_{\frac{n}{m}} u=\left\{
\begin{split}
&\nabla.\{\Delta^{j-1}(|\nabla \Delta^{j-1} u|^{\frac{n}{m}-2} \nabla \Delta^{j-1}u)\} \; \ \  \text{if}\ m=2j-1;  \\
&\Delta^{j}(|\Delta^j u|^{\frac{n}{m}-2} \Delta^j u) \;\ \ \ \ \ \ \ \ \ \ \ \ \ \ \ \ \ \ \ \ \ \ \text{if}\ m=2j.
\end{split}
\right.
\end{equation*}
The symbol $\nabla^m u $ denotes the $m^{\text{th}}$ order gradient of $u$ and is defined as,
\begin{equation*}
\nabla^m u=\left\{
\begin{split}
&\nabla \Delta^{(m-1)/2}u \; \  \text{if}\ m \ \text{is odd};  \\
&\Delta^{m/2} u \;\ \ \ \ \ \ \ \ \text{if}\ m \ \text{is even},
\end{split}
\right.
\end{equation*}
where $\Delta$ and $\nabla$ denotes the usual Laplacian and gradient operator respectively and also $\nabla^m u. \nabla^m v$ denotes the product of two vectors when $m$ is odd and product of two scalars when $m$ is even.\vspace{0.2cm}\\
Using Theorems \ref{MTST0}, \ref{MTST01} and extension of Theorems \ref{MTST00} and \ref{MTST11} (which is an open question), we can study the system of Kirchhoff-Choquard equation for the Polyharmonic operator. \vspace{0.2cm}\\
\textbf{3}: Another important open question is the Adams-Moser-Trundinger inequalities in Cartesian product of Sobolev space with unbounded domain (or in $\mathbb{R}^n$).

\section*{Acknowledgement}
The first author would like to thank Department of Mathematics, Indian Institute of Technology, Delhi for their kind hospitality during his visit.

\end{document}